\documentclass[a4paper,11pt]{article}
\usepackage{amsmath, amsfonts, amscd, amssymb, amsthm, enumerate, xypic, float}

\newcommand{\nfix}{\operatorname{nfix}}
\newcommand{\car}{\chi}

\newcommand{\Mat}{\operatorname{M}}

\newcommand{\Pgros}{\operatorname{\mathbb{P}}}
\newcommand{\id}{\operatorname{id}}
\newcommand{\GL}{\operatorname{GL}}

\newcommand{\Diag}{\operatorname{Diag}}
\newcommand{\Vect}{\operatorname{span}}
\newcommand{\im}{\operatorname{Im}}

\newcommand{\tr}{\operatorname{tr}}

\newcommand{\Aut}{\operatorname{Aut}}

\newcommand{\per}{\operatorname{per}}
\newcommand{\sgn}{\operatorname{sgn}}
\newcommand{\codim}{\operatorname{codim}}
\renewcommand{\setminus}{\smallsetminus}

\def\F{\mathbb{F}}

\renewcommand{\L}{\mathbb{L}}

\def\calC{\mathcal{C}}

\def\calF{\mathcal{F}}
\def\calG{\mathcal{G}}

\def\calO{\mathcal{O}}

\def\calS{\mathcal{S}}

\def\calV{\mathcal{V}}
\def\calW{\mathcal{W}}

\def\lcro{\mathopen{[\![}}
\def\rcro{\mathclose{]\!]}}

\theoremstyle{definition}
\newtheorem{Def}{Definition}[section]
\newtheorem{Not}[Def]{Notation}

\theoremstyle{plain}
\newtheorem{theo}{Theorem}[section]
\newtheorem{prop}[theo]{Proposition}
\newtheorem{cor}[theo]{Corollary}
\newtheorem{lemma}[theo]{Lemma}
\newtheorem{claim}{Claim}

\theoremstyle{plain}
\newtheorem{conj}{Conjecture}

\theoremstyle{remark}
\newtheorem{Rems}{Remarks}
\newtheorem{Rem}[Rems]{Remark}
\newtheorem{ex}{Example}

\title{On the linear preservers of Schur matrix functionals}
\author{Cl\'ement de Seguins Pazzis\footnote{Universit\'e de Versailles Saint-Quentin-en-Yvelines, Laboratoire de Math\'ematiques
de Versailles, 45 avenue des Etats-Unis, 78035 Versailles cedex, France}
\footnote{e-mail address: dsp.prof@gmail.com}}

\begin{document}

\thispagestyle{plain}


\maketitle

\begin{abstract}
Let $\F$ be a field and $f : \mathfrak{S}_n \rightarrow \F \setminus \{0\}$ be an arbitrary map.
The Schur matrix functional associated to $f$
is defined as $M \in \Mat_n(\F) \mapsto \widetilde{f}(M):=\sum_{\sigma \in \mathfrak{S}_n} f(\sigma) \prod_{j=1}^n m_{\sigma(j),j}$.
Typical examples of such functionals are the determinant (where $f$ is the signature morphism)
and the permanent (where $f$ is constant with value $1$).
Given two such maps $f$ and $g$, we study the endomorphisms $U$ of the vector space $\Mat_n(\F)$
that satisfy $\widetilde{g}(U(M))=\widetilde{f}(M)$ for all $M \in \Mat_n(\F)$.
In particular, we give a closed form for the linear preservers of the functional $\widetilde{f}$
when $f$ is central, and as a special case we extend to an arbitrary field Botta's characterization of the linear preservers of the permanent.
\end{abstract}




\section{Introduction}

\subsection{Notation}

Throughout, we fix an arbitrary field $\F$ whose characteristic we denote by $\car(\F)$, and whose
group of non-zero elements we denote by $\F^*$.
Let $n$ and $p$ be non-negative integers.
We denote by $\Mat_{n,p}(\F)$ the vector space of all $n$ by $p$ matrices with entries in $\F$.
In particular, we denote by $\Mat_n(\F):=\Mat_{n,n}(\F)$ the algebra of all square $n$ by $n$ matrices with entries in $\F$, and by $\GL_n(\F)$ its group of invertible elements. The rows of a matrix $M \in \Mat_n(\F)$ are denoted by $R_1(M),\dots,R_n(M)$, and its columns by
$C_1(M),\dots,C_n(M)$.
We denote by $E_{i,j}$ the matrix unit of $\Mat_{n,p}(\F)$ with zero entries everywhere except at the $(i,j)$-spot where the entry equals $1$.
Given matrices $A=(a_{i,j})$ and $B=(b_{i,j})$ in $\Mat_n(\F)$, we denote their Hadamard product by
$$A \star B:=(a_{i,j}\,b_{i,j})_{1 \leq i,j \leq n.}$$
We set
$$E:=(1)_{1 \leq i,j\leq n},$$
so that $A \star E=E \star A$ for all $A \in \Mat_n(\F)$.
Given $A \in \Mat_n(\F^*)$, we set
$$A^{[-1]}:=(1/a_{i,j})_{1 \leq i,j \leq n},$$
so that $A \star A^{[-1]}=A^{[-1]} \star A=E$.

Denote by $\mathfrak{S}_n$ the group of all permutations of $\lcro 1,n\rcro$ and by $\mathfrak{A}_n$
the corresponding alternating group. Given distinct elements $i_1,\dots,i_p$ of $\lcro 1,n\rcro$ (with $p>1$),
we denote by $(i_1\,i_2\,\cdots\, i_p)$ the $p$-cycle that fixes every element of $\lcro 1,n\rcro \setminus \{i_1,\dots,i_p\}$,
maps $i_k$ to $i_{k+1}$ for all $k \in \lcro 1,p-1\rcro$, and maps $i_p$ to $i_1$.
Given distinct elements $i,j$ of $\lcro 1,n\rcro$, the transposition $(i\; j)$ is also denoted by $\tau_{i,j}$.

Given a permutation $\sigma$ of $\lcro 1,n\rcro$, the associated \textbf{permutation matrix} in $\Mat_n(\F)$
is denoted by
$$P_{\sigma}:=(\delta_{i,\sigma(j)})_{1 \leq i,j \leq n.}$$

\subsection{The problem}

Linear preservers are a standard topic in modern linear algebra.
One of the first results of this kind was Frobenius's determination of the linear bijections from $\Mat_n(\F)$
to itself that preserve the determinant \cite{Frobenius} (when $\F$ is an infinite field).
In the present work, we are concerned with a generalization of Frobenius's result to a wider class of matrix functionals.

To any scalar-valued mapping $f : \mathfrak{S}_n \rightarrow \F$,
we associate its Schur functional on square matrices, defined as follows:
$$\widetilde{f} : M=(m_{i,j}) \in \Mat_n(\F) \longmapsto \underset{\sigma \in \mathfrak{S}_n}{\sum}
f(\sigma)\,\prod_{j=1}^n m_{\sigma(j),j.}$$
Throughout, we will assume that $f$ vanishes \emph{nowhere}. This discards a lot of interesting functions,
but it appears that this assumption is key to the success of our methods.

Note that $f(\sigma)=\widetilde{f}(P_\sigma)$ for all $\sigma \in \mathfrak{S}_n$, and hence $\widetilde{f}$ determines $f$.
Two standard examples of such maps are the determinant (where $f$ is the signature morphism, denoted by $\sgn$)
and the permanent, denoted by $\per$ (where $f$ is constant with value $1$).
Here, we wish to find a closed form for the linear maps
$$U : \Mat_n(\F) \rightarrow \Mat_n(\F)$$
such that
\begin{equation}\label{preserve}
\forall M \in \Mat_n(\F), \; \widetilde{f}\bigl(U(M)\bigr)=\widetilde{f}(M).
\end{equation}
More generally, we are interested, given potentially different mappings $f : \mathfrak{S}_n \rightarrow \F^*$
and $g : \mathfrak{S}_n \rightarrow \F^*$, in the endomorphisms $U$ of the vector space $\Mat_n(\F)$ that satisfy
$$\forall M \in \Mat_n(\F), \; \widetilde{g}(U(M))=\widetilde{f}(M).$$
Such endomorphisms will be called \textbf{$(f,g)$-transformations}. We want to know if such transformations exist, and in the affirmative
we want to find a closed form for them.

Similar questions have been studied in the past, see \cite{Botta1,Botta2,Duffner}.
The situation is slightly different here: we do not allow our functions $f$ and $g$ to take the value $0$
(this will prove crucial in some key lemmas), whereas in the works we have just cited
the functions under hand tend to vanish at many permutations. On the other hand, in the known works
on the topic the function $f$ is much more specific than what we have in mind: in \cite{Botta1}
the support of $f$ is a transitive cyclic subgroup of $\mathfrak{S}_n$; in \cite{Botta2},
the support of $f$ is a doubly-transitive and regular proper subgroup $G$ of $\mathfrak{S}_n$, and the restriction of
$f$ to this subgroup is a group homomorphism to the multiplicative group $\F^*$;
in \cite{Duffner} the field $\F$ is the one of complex numbers and $f$ is an irreducible character of degree greater than $1$ on
$\mathfrak{S}_n$.

Let us recall some known results:

\begin{theo}[Frobenius \cite{Frobenius}, Dieudonn\'e \cite{Dieudonne}]\label{detpreservers}
If $f=\sgn$, then the linear maps $U : \Mat_n(\F) \rightarrow \Mat_n(\F)$
that satisfy \eqref{preserve} are the maps of the form
$$M \mapsto PMQ \quad \text{or} \quad M \mapsto PM^T Q$$
for some pair $(P,Q)\in \GL_n(\F)^2$ such that $\det(P)\det(Q)=1$.
\end{theo}

For the permanent, the preservers have a much more rigid form:

\begin{theo}[Botta \cite{Botta3}, Marcus and May \cite{MarcusMay}]\label{perpreservers}
Assume that $|\F| \geq n \geq 3$ and $\car(\F) \neq 2$.
If $f$ is constant with value $1$, then the linear maps $U : \Mat_n(\F) \rightarrow \Mat_n(\F)$
that satisfy \eqref{preserve}
are the maps of the form
$$M \mapsto R \star (P_\sigma MP_\tau) \quad \text{or} \quad M \mapsto R \star (P_\sigma M^T P_\tau)$$
where $(\sigma,\tau)\in (\mathfrak{S}_n)^2$ and $R \in \Mat_n(\F)$ is a rank $1$ matrix
the product of whose diagonal entries equals $1$.
\end{theo}

If $\F$ has characteristic $2$, then the permanent is just the determinant.
For $n=2$, the permanent is deduced from the determinant through a linear bijection:
in that case indeed, we have
$$\forall M \in \Mat_2(\F), \; \per(M)=\det(K \star M) \quad \text{where}\;
K:=\begin{bmatrix}
1 & -1 \\
1 & 1
\end{bmatrix}.$$
Hence, for $n=2$ permanent preservers are easily deduced from determinant preservers.

\begin{Rem}\label{n=2remark}
More generally, we note that if $n=2$ then by setting $\alpha:=\frac{f(\tau)}{f(\id)}$ where $\tau$ is the transposition of $\mathfrak{S}_2$,
we find
$$\widetilde{f} : M \mapsto f(\id) \det(A \star M) \quad \text{for} \quad
A:=\begin{bmatrix}
1 & -\alpha \\
1 & 1
\end{bmatrix}.$$
In particular, the linear maps that satisfy \eqref{preserve} are easily deduced from the
linear preservers of the determinant.
\end{Rem}

In this article, our aim is to generalize the above results to an arbitrary
mapping $f$ with no restriction on the cardinality of the underlying field.
In particular, we will generalize Theorem \ref{perpreservers}
to an arbitrary field of characteristic different from $2$.

The following result, which we shall prove right away, is relevant to Theorem \ref{perpreservers}:

\begin{lemma}\label{HadProductLemma}
Let $R \in \Mat_n(\F)$. Set $U : M \in \Mat_n(\F) \mapsto R \star M$.
Then, the following conditions are equivalent:
\begin{enumerate}[(i)]
\item The mapping $U$ is an $(f,f)$-transformation.
\item The matrix $R$ has rank $1$ and the product of its diagonal entries equals $1$.
\item There are column matrices $X=(x_k)$ and $Y=(y_k)$ in $\F^n$ such that $R=XY^T$ and $\underset{k=1}{\overset{n}{\prod}} x_ky_k=1$.
\end{enumerate}
A rank $1$ matrix that satisfies condition (ii) will be called \textbf{normalized}.
\end{lemma}

\begin{proof}
Throughout the proof, we write $R=(r_{i,j})_{1 \leq i,j \leq n.}$
Classically, a matrix of $\Mat_n(\F)$ has rank $1$ if and only if it reads $XY^T$ for some nonzero matrices $X,Y$ in $\F^n$,
and it easily follows that conditions (ii) and (iii) are equivalent.

Next, note that
$$\forall M \in \Mat_n(\F), \; \widetilde{f}(U(M))=\sum_{\sigma \in \mathfrak{S}_n} f(\sigma) \prod_{j=1}^n r_{\sigma(j),j}\,m_{\sigma(j),j}.$$
It follows that $U$ is an $(f,f)$-transformation if and only if
\begin{equation}\label{Crank1}
\forall \sigma \in \mathfrak{S}_n, \; \prod_{j=1}^n r_{\sigma(j),j}=1
\end{equation}
(for the direct implication, take for $M$ any permutation matrix).
From there, it is obvious that condition (iii) implies condition (i).

Conversely, assume that condition (i) holds, so that \eqref{Crank1} also holds. Note in particular that all the entries of
$R$ are nonzero.
The case $\sigma=\id$ shows that the product of the diagonal entries of $R$ equals $1$.
It remains to prove that $R$ has rank $1$, i.e.\ that all the $2$ by $2$ minors of $R$ vanish.
To this end, let $i_1,i_2,j_1,j_2$ be indices in $\lcro 1,n\rcro$ such that $i_1<i_2$ and $j_1<j_2$.
Choose $\sigma \in \mathfrak{S}_n$ such that $\sigma(j_k)=i_k$ for all $k \in \{1,2\}$. Denote finally by $\tau$ the transposition of $\lcro 1,n\rcro$
that exchanges $i_1$ and $i_2$. Applying \eqref{Crank1} to $\tau \sigma$ and $\sigma$ yields
$$r_{i_2,j_1}r_{i_1,j_2} \prod_{k \in \lcro 1,n\rcro \setminus \{j_1,j_2\}}r_{\sigma(k),k}=r_{i_2,j_2}r_{i_1,j_1}\prod_{k \in \lcro 1,n\rcro \setminus \{j_1,j_2\}}r_{\sigma(k),k},$$
whence $r_{i_2,j_1}r_{i_1,j_2} =r_{i_2,j_2}r_{i_1,j_1}$.
Hence, all the $2$ by $2$ minors of $R$ vanish, and we conclude that $R$ has rank $1$. Thus, we have shown that condition (i) implies condition (ii), which completes the proof.
\end{proof}

\subsection{The key equivalence relations on matrix functionals}

We define a right-action of the group $(\Mat_n(\F^*),\star)$
on the set of all maps from  $\mathfrak{S}_n$ to $\F^*$ as follows:
given $A \in \Mat_n(\F^*)$ and $f  : \mathfrak{S}_n \rightarrow \F^*$, we define
$f.A$ as
$$g : \sigma \mapsto f(\sigma)\prod_{k=1}^n a_{\sigma(k),k}$$
or, alternatively, as the map from $\mathfrak{S}_n$ to $\F^*$ whose associated Schur functional is
$$M \mapsto \widetilde{f}(A\star M).$$

Two maps from $\mathfrak{S}_n$ to $\F^*$ are called \textbf{H-equivalent}\footnote{The letter ``H" stands for ``Hadamard product".}
when they belong to the same orbit under the above action of $(\Mat_n(\F^*),\star)$.

\begin{Rem}\label{HequivalentScalarProduct}
Let $\alpha \in \F^*$. Then, $\alpha f$ is H-equivalent to $f$. Indeed,
by taking $A=(a_{i,j}) \in \Mat_n(\F)$ as the matrix in which all the entries in the first column
equal $\alpha$, and all the other ones equal $1$, we have
$\widetilde{f}(A \star M)=\widetilde{\alpha f}(M)$ for all $M \in \Mat_n(\F)$.

Combining this with Remark \ref{n=2remark} yields that every mapping from $\mathfrak{S}_2$ to $\F^*$
is H-equivalent to the signature.
\end{Rem}

Next, we consider the semi-direct product $\Mat_n(\F^*) \rtimes (\mathfrak{S}_n)^2$ associated with the group homomorphism
$$(\tau, \tau') \in (\mathfrak{S}_n)^2 \mapsto (A \mapsto P_\tau A P_{\tau'}^{-1}) \in \Aut(\Mat_n(\F^*),\star),$$
and we define a right-action of this semi-direct product on the function set $\calF(\mathfrak{S}_n,\F^*)$ as follows:
given a triple $(A,\tau,\tau') \in \Mat_n(\F^*) \rtimes (\mathfrak{S}_n)^2$ and a map
$f : \mathfrak{S}_n \rightarrow \F^*$, we define
$f.(A,\tau,\tau')$ as the map
$$\sigma \mapsto f(\tau \sigma \tau'^{-1})\prod_{k=1}^n a_{(\tau \sigma \tau'^{-1})(k),k}$$
or, alternatively, as the map whose associated functional
is
$$M \mapsto \widetilde{f}\bigl(A\star (P_\tau M P_{\tau'}^{-1})\bigr).$$

Two maps from $\mathfrak{S}_n$ to $\F^*$ are called \textbf{PH-equivalent}\footnote{The letter ``P" stands for ``permutation".}
whenever they belong to the same orbit under the above action of $\Mat_n(\F^*)\rtimes (\mathfrak{S}_n)^2$.

\vskip 3mm
Here is another important construction: to any $f : \mathfrak{S}_n \rightarrow \F^*$, we associate its \textbf{transpose}
$$f^T: \sigma \in \mathfrak{S}_n \mapsto f(\sigma^{-1}).$$
One checks that the functional $\widetilde{f^T}$ is no other than $M \mapsto \widetilde{f}(M^T)$.

\subsection{Main results, and structure of the article}

We can now state some of the main results of this article.

\begin{theo}\label{automtheo}
Let $f$ and $g$ be maps from $\mathfrak{S}_n$ to $\F^*$.
Every $(f,g)$-transformation is an automorphism of the vector space $\Mat_n(\F)$.
\end{theo}

As a corollary, every $(f,g)$-transformation $U$ is invertible and its inverse is a $(g,f)$-transformation.

Consider the category $\calC$ whose objects are the maps from $\mathfrak{S}_n$ to $\F^*$
and in which, given two such objects $f$ and $g$, the morphisms from $f$ to $g$ are the $(f,g)$-transformations
(with the composition of morphisms defined as the composition of endomorphisms of $\Mat_n(\F)$).
If follows from the above result that $\calC$ is actually a \emph{groupoid}. Hence,
in order to determine the morphisms in $\calC$, it suffices to answer the following questions:
\begin{itemize}
\item Given two objects $f$ and $g$, when does there exist an $(f,g)$-transformation?
\item Given an object $f$, what are the $(f,f)$-transformations?
\end{itemize}
Moreover, to answer the second question for a specific $f$, it suffices to answer it for a well-chosen
$g$ for which there exists an $(f,g)$-transformation.

The next theorem yields a full answer to the first question:

\begin{theo}
Let $f$ and $g$ be maps from $\mathfrak{S}_n$ to $\F^*$.
The following conditions are equivalent:
\begin{enumerate}[(i)]
\item An $(f,g)$-transformation exists.
\item The mapping $g$ is PH-equivalent to $f$ or to $f^T$.
\end{enumerate}
\end{theo}

As far as the second question is concerned, we will give a partial answer to it in Section \ref{transfosection}.
This answer cannot be stated at this point of the article because it involves important objects that are attached to $f$
and that are studied in Section \ref{normalizeSection}, namely the column and row partitions of $f$.
The answer is only partial because we do not reach a completely closed form for the $(f,f)$-transformations: we will
prove that the $(f,f)$-transformations have a certain form, but not all maps of the given form are $(f,f)$-transformations in general.
However, in the special case when $f$ is central, i.e.\ constant on every conjugacy class of the group $\mathfrak{S}_n$,
 we will give a closed form for the $(f,f)$-transformations (Section \ref{centralsection}).
In particular, we will generalize Theorem \ref{perpreservers} to an arbitrary field with characteristic not $2$
(with no restriction of cardinality).

Our approach to the study of $(f,g)$-transformations is a traditional one that dates back to Dieudonn\'e:
first, one determines the linear subspaces of $\Mat_n(\F)$ with the maximal dimension among those
that are included in the null cone of $\widetilde{f}$ (Section \ref{conesection}); those spaces are deeply connected to the row and column partitions associated with $f$; then, given an $(f,g)$-transformation $U$, one considers the inverse image under $U$ of such a subspace attached to $\widetilde{g}$, which yields precious information on $U$ (see Section \ref{transfosection}).

\section{Normalized mappings}\label{normalizeSection}

In this section, we introduce the row and column partitions of a mapping $f : \mathfrak{S}_n \rightarrow \F^*$.
Then, we show that $f$ is always PH-equivalent to a specific type of mapping called \emph{fully-normalized}.
Fully-normalized mappings are important because their linear preservers are much more easily expressed than in the general case.

\subsection{The column and row partitions attached to $f$}

\begin{lemma}\label{BasicEquivalenceLemma}
Let $i,j$ be distinct elements of $\lcro 1,n\rcro$, and $Z=(z_k)_{1 \leq k \leq n} \in (\F^*)^n$.
The following conditions are equivalent:
\begin{itemize}
\item[(i)] The map $\widetilde{f}$ vanishes at every $M \in \Mat_n(\F)$ such that $C_j(M)=Z \star C_i(M)$.
\item[(ii)] One has $f(\sigma   \tau_{i,j})=-\frac{z_{\sigma(j)}}{z_{\sigma(i)}} f(\sigma)$ for all $\sigma \in \mathfrak{S}_n$.
\end{itemize}
The following conditions are also equivalent:
\begin{itemize}
\item[(iii)] The map $\widetilde{f}$ vanishes at every $M \in \Mat_n(\F)$ such that $R_j(M)=Z^T \star R_i(M)$.
\item[(iv)] One has $f(\tau_{i,j}   \sigma)=-\frac{z_{\sigma^{-1}(j)}}{z_{\sigma^{-1}(i)}} f(\sigma)$ for all $\sigma \in \mathfrak{S}_n$.
\end{itemize}
\end{lemma}

\begin{proof}
Assume that condition (ii) holds. Let $M=(m_{k,l})_{1 \leq k,l\leq n} \in \Mat_n(\F)$ be such that $C_j(M)=Z \star C_i(M)$.
Then,
\begin{align*}
\widetilde{f}(M) & =\sum_{\sigma \in \mathfrak{A}_n} \Bigl(f(\sigma)\prod_{k=1}^n m_{\sigma(k),k}+f(\sigma   \tau_{i,j})
\prod_{k=1}^n m_{(\sigma   \tau_{i,j})(k),k}\Bigr) \\
& = \sum_{\sigma \in \mathfrak{A}_n} f(\sigma)\biggl(\prod_{k=1}^n m_{\sigma(k),k}-\frac{z_{\sigma(j)}}{z_{\sigma(i)}}
\prod_{k=1}^n m_{\sigma(\tau_{i,j}(k)),k}\biggr) \\
& = \sum_{\sigma \in \mathfrak{A}_n} f(\sigma) \biggl(\prod_{k \in \lcro 1,n\rcro \setminus \{i,j\}} m_{\sigma(k),k}\biggr)\,
\biggl(m_{\sigma(i),i}\, m_{\sigma(j),j}-\frac{z_{\sigma(j)}}{z_{\sigma(i)}}\,m_{\sigma(j),i}\,m_{\sigma(i),j}\biggr) \\
& = \sum_{\sigma \in \mathfrak{A}_n} f(\sigma) \biggl(\prod_{k \in \lcro 1,n\rcro \setminus \{i,j\} }m_{\sigma(k),k}\biggr)\,
\biggl(m_{\sigma(i),i}\, z_{\sigma(j)}\, m_{\sigma(j),i}-\frac{z_{\sigma(j)}}{z_{\sigma(i)}}\,m_{\sigma(j),i}\, z_{\sigma(i)}\, m_{\sigma(i),i}\biggr) \\
& =0.
\end{align*}

Conversely, assume that condition (i) holds.
Let $\sigma \in \mathfrak{S}_n$. Consider the matrix
$M=(m_{k,l})_{1 \leq k,l \leq n}$ of $\Mat_n(\F)$ defined as follows:
$$m_{k,l}=\begin{cases}
1 & \text{if $k=\sigma(l)$ and $l \not\in \{i,j\}$} \\
1 & \text{if $k\in \{\sigma(i),\sigma(j)\}$ and $l=i$} \\
z_{\sigma(i)} & \text{if $(k,l)=(\sigma(i),j)$} \\
z_{\sigma(j)} & \text{if $(k,l)=(\sigma(j),j)$} \\
0 & \text{otherwise.}
\end{cases}$$
On the one hand, we have $C_j(M)=Z \star C_i(M)$, whence $\widetilde{f}(M)=0$.
On the other hand, one sees that for all $\sigma' \in \mathfrak{S}_n$, one has
$\underset{k=1}{\overset{n}{\prod}} m_{\sigma'(l),l}=0$ whenever $\sigma' \neq \sigma$ and $\sigma' \neq \sigma   \tau_{i,j}$.
Hence,
$$f(\sigma) \prod_{l=1}^n m_{\sigma(l),l}+f(\sigma   \tau_{i,j}) \prod_{l=1}^n m_{(\sigma   \tau_{i,j})(l),l}=0,$$
which reads
$$z_{\sigma(j)} f(\sigma)+z_{\sigma(i)} f(\sigma  \tau_{i,j})=0.$$
Therefore, condition (ii) is satisfied.

To obtain that conditions (iii) and (iv) are equivalent, we note that
$M \mapsto \widetilde{f}(M^T)$ is the matrix functional associated with $f^T$ and we apply
the equivalence between conditions (i) and (ii) to this functional.
\end{proof}

\begin{lemma}[Transitivity lemma]\label{Transitivitylemma}
Let $i,j,k$ be distinct indices in $\lcro 1,n\rcro$, and $Z$ and $Z'$ be vectors of $(\F^*)^n$.
Assume that:
\begin{enumerate}[(i)]
\item $\widetilde{f}$ vanishes at every matrix $M \in \Mat_n(\F)$ such that $C_j(M)=Z \star C_i(M)$;
\item $\widetilde{f}$ vanishes at every matrix $M \in \Mat_n(\F)$ such that $C_k(M)=Z' \star C_j(M)$.
\end{enumerate}
Then, $\widetilde{f}$ vanishes at every matrix $M \in \Mat_n(\F)$ such that $C_k(M)=(Z \star Z') \star C_i(M)$.
\end{lemma}

\begin{proof}
By Lemma \ref{BasicEquivalenceLemma}, we have $f(\sigma   \tau_{i,j})=-\frac{z_{\sigma(j)}}{z_{\sigma(i)}} f(\sigma)$ and
$f(\sigma   \tau_{j,k})=-\frac{z'_{\sigma(k)}}{z'_{\sigma(j)}} f(\sigma)$ for all $\sigma$ in $\mathfrak{S}_n$.
Let $\sigma \in \mathfrak{S}_n$. Then,
\begin{align*}
f(\sigma \tau_{i,k}) & = f(\sigma \tau_{j,k} \tau_{i,j} \tau_{j,k}) \\
& = -f(\sigma \tau_{j,k} \tau_{i,j}) \frac{z'_{\sigma(j)}}{z'_{\sigma(i)}} \\
& =  f(\sigma \tau_{j,k}) \frac{z_{\sigma(k)}}{z_{\sigma(i)}} \frac{z'_{\sigma(j)}}{z'_{\sigma(i)}} \\
& = -f(\sigma) \frac{z'_{\sigma(k)}}{z'_{\sigma(j)}} \frac{z_{\sigma(k)}}{z_{\sigma(i)}} \frac{z'_{\sigma(j)}}{z'_{\sigma(i)}} \\
& = -f(\sigma) \frac{z_{\sigma(k)} z'_{\sigma(k)}}{z_{\sigma(i)}z'_{\sigma(i)}}\cdot
\end{align*}
The conclusion ensues, by Lemma \ref{BasicEquivalenceLemma}.
\end{proof}

\begin{Def}
Let $i,j$ be elements of $\lcro 1,n\rcro$.

We say that $i$ is \textbf{column-$f$-equivalent} to $j$, and we write $i \underset{C,f}{\sim} j$, when either
$i=j$, or $i \neq j$ and there exists a vector $Z \in (\F^*)^n$ such that
$\widetilde{f}(M)=0$ for all $M \in \Mat_n(\F)$ satisfying $C_j(M)=Z \star C_i(M)$.

We say that $i$ is \textbf{row-$f$-equivalent} to $j$, and we write $i \underset{R,f}{\sim} j$, when either
$i=j$, or $i \neq j$ and there exists a vector $Z \in (\F^*)^n$ such that
$\widetilde{f}(M)=0$ for all $M \in \Mat_n(\F)$ satisfying $R_j(M)=Z^T \star R_i(M)$.
\end{Def}

Using the transpose of $f$, one sees that $i$ is row-$f$-equivalent to $j$ if and only if it is column-$f^T$-equivalent to $j$.

\begin{prop}
The relations of column-$f$-equivalence and row-$f$-equivalence are equivalence relations on $\lcro 1,n\rcro$.
\end{prop}

\begin{proof}
Let us consider column-$f$-equivalence.

First of all, column-$f$-equivalence is reflexive by definition.
Next, given distinct indices $i,j$ in $\lcro 1,n\rcro$ such that $i \underset{C,f}{\sim} j$, we have a vector
$Z=(z_k) \in (\F^*)^n$ such that $\widetilde{f}(M)=0$ for all $M \in \Mat_n(\F)$ satisfying $C_j(M)=Z \star C_i(M)$.
For all $M \in \Mat_n(\F)$ such that $C_i(M)=Z^{[-1]} \star C_j(M)$,
we have $C_j(M)=Z\star C_i(M)$ and hence $\widetilde{f}(M)=0$. Therefore, $j \underset{C,f}{\sim} i$.

Finally, let $i,j,k$ be elements of $\lcro 1,n\rcro$, and assume that $i \underset{C,f}{\sim} j$ and $j \underset{C,f}{\sim} k$.
If $i=j$ or $j=k$ or $i=j$ then it is obvious that $i$ is column-$f$-equivalent to $k$. If $i,j,k$ are pairwise distinct, then
Lemma \ref{Transitivitylemma} shows that $i \underset{C,f}{\sim} k$.

We conclude that column-$f$-equivalence is an equivalence relation on $\lcro 1,n\rcro$.
It follows that row-$f$-equivalence, which is simply column-$f^T$-equivalence, is also an equivalence relation on $\lcro 1,n\rcro$.
\end{proof}

\begin{Def}
A mapping $f : \mathfrak{S}_n \rightarrow \F^*$ is called \textbf{rigid} when
its column equivalence classes and its row equivalence classes are singletons.
\end{Def}

\begin{Rem}
If $i$ and $j$ are distinct column-$f$-equivalent indices, there is a vector $Z \in (\F^*)^n$ such that
$$\forall \sigma \in \mathfrak{S}_n, \; f(\sigma \tau_{i,j})=-\frac{z_{\sigma(j)}}{z_{\sigma (i)}} f(\sigma).$$
It is then obvious that this condition determines $Z$ up to multiplication by a non-zero scalar.
\end{Rem}

Now, we show that any two H-equivalent functionals determine the same column-equivalence and row-equivalence relations.

\begin{lemma}\label{Hequivalencelemma}
Let $A \in \Mat_n(\F^*)$.
Set $\widetilde{g} : M \in \Mat_n(\F) \mapsto \widetilde{f}(A \star M)$.
Let $i,j$ be distinct indices in $\lcro 1,n\rcro$, and $Z$ be a vector of $(\F^*)^n$ such that
$\widetilde{f}(M)=0$ for all $M \in \Mat_n(\F)$ satisfying $C_j(M)=Z \star C_i(M)$.
Set $Z':=Z \star C_i(A) \star C_j(A)^{[-1]}\in (\F^*)^n$. Then,
$\widetilde{g}(M)=0$ for all $M \in \Mat_n(\F)$ such that $C_j(M)=Z' \star C_i(M)$.
\end{lemma}

\begin{proof}
Let $M \in \Mat_n(\F)$ be such that $C_j(M)=Z' \star C_i(M)$.
Then, $C_j(A \star M)=C_j(A) \star C_j(M)=(C_j(A) \star Z') \star C_i(M)=(Z \star C_i(A)) \star C_i(M)=Z \star C_i(A \star M)$.
Hence, $\widetilde{g}(M)=\widetilde{f}(A \star M)=0$.
\end{proof}

\begin{cor}\label{HequivalenceCorollary}
Let $g : \mathfrak{S}_n \rightarrow \F^*$ be a mapping that is H-equivalent to $f$.
Then, two indices $i$ and $j$ in $\lcro 1,n\rcro$ are column-$f$-equivalent (respectively, row-$f$-equivalent) if and only if they are column-$g$-equivalent (respectively, row-$g$-equivalent).
\end{cor}

We now look at two basic examples:

\begin{ex}\label{detExample}
Assume that $\widetilde{f}=\det$. Here, $f$ is the signature morphism, and it follows that
any two indices in $\lcro 1,n\rcro$ are column-$f$-equivalent and row-$f$-equivalent.
\end{ex}

\begin{ex}\label{perExample}
Assume that $\widetilde{f}=\per$, that $\car(\F) \neq 2$ and that $n \geq 3$.
We claim that $f$ is rigid.

Here, $f$ is the constant map with value $1$.
Assume that there are distinct indices $i$ and $j$ that are column-$f$-equivalent, and let
$Z=(z_k)_{1 \leq k \leq n}$ be an associated vector of $(\F^*)^n$. Let
$k,l$ be distinct indices in $\lcro 1,n\rcro$. We can find a permutation $\sigma \in \mathfrak{S}_n$ such that $\sigma(i)=k$ and $\sigma(j)=l$.
Hence, $z_l=-z_k$. As $n \geq 3$ we can find an index $u \in \lcro 1,n\rcro \setminus \{k,l\}$, and hence
$z_k=-z_u=z_l$. Since $\car(\F) \neq 2$, it follows that $z_k=0$, contradicting our assumptions.

We conclude that the equivalence classes for column-$f$-equivalence are singletons. Likewise, the ones for
row-$f$-equivalence are singletons.
\end{ex}

\begin{ex}
Assume that $n=2$. Then, $1$ and $2$ are both row-$f$-equivalent and column-$f$-equivalent.
Indeed, we have seen in Remark \ref{n=2remark} that $f$ is H-equivalent to the signature morphism, and
by Example \ref{detExample} and Corollary \ref{HequivalenceCorollary}, it follows that any two indices in $\lcro 1,2\rcro$
are column-$f$-equivalent and row-$f$-equivalent.
\end{ex}

In Section \ref{possibleclassesSection}, we will examine more closely what the column-equivalence classes (or row-equivalence classes)
can be.

\subsection{Normalized and fully-normalized functionals}

\begin{Def}
We say that $f$ is \textbf{column-normalized} (respectively, \textbf{row-normalized})
when, for all column-$f$-equivalent distinct indices $i$ and $j$ in $\lcro 1,n\rcro$,
one has $\forall \sigma \in \mathfrak{S}_n, \; f(\sigma \tau_{i,j})=-f(\sigma)$
(respectively, $\forall \sigma \in \mathfrak{S}_n, \; f(\tau_{i,j} \sigma )=-f(\sigma)$).

We say that $f$ is \textbf{normalized} when it is both column-normalized and row-normalized.

We say that $f$ is \textbf{fully-normalized} when it is normalized and it satisfies the following additional conditions:
\begin{enumerate}[(a)]
\item The column-$f$-equivalence classes are intervals of integers, in non-increasing order of cardinality
(i.e.\ for all $(i,j)\in \lcro 1,n\rcro^2$ with $i<j$, the cardinality of the column-$f$-equivalence class of
$i$ is greater than or equal to the one of the column-$f$-equivalence class of $j$).
\item The row-$f$-equivalence classes are intervals of integers, in non-increasing order of cardinality.
\end{enumerate}
\end{Def}

\begin{Rem}\label{detremark}
Assume that $f$ is column-normalized and that column-$f$-equivalence is trivial (that is, any two indices are column-$f$-equivalent).
Then, by identifying a matrix with the list of its columns, we see $\widetilde{f}$ as an alternating $n$-linear form on $(\F^n)^n$,
and hence classically $\widetilde{f}=\lambda\,\det$ for some scalar $\lambda$ (with $\lambda \neq 0$ because $\widetilde{f} \neq 0$).
Conversely, if $\widetilde{f}$ is a scalar multiple of the determinant, then it is normalized, and column-$f$-equivalence and row-$f$-equivalence are trivial (i.e.\ they relate all the elements of $\lcro 1,n\rcro$).
\end{Rem}

Now, we turn to the most important result of the present section:

\begin{prop}\label{normalizationProp}
Every matrix functional is H-equivalent to a normalized one.
\end{prop}

Proving this result requires an additional lemma:

\begin{lemma}\label{forcingrowlemma}
Assume that $f$ is column-normalized. Let $i$ and $j$ be distinct row-$f$-equivalent indices
in $\lcro 1,n\rcro$, and $Z=(z_k) \in (\F^*)^n$ be such that
$$\forall \sigma \in \mathfrak{S}_n, \; f(\tau_{i,j} \sigma)=-\frac{z_{\sigma^{-1}(j)}}{z_{\sigma^{-1}(i)}}\,f(\sigma).$$
Then, $z_k=z_l$ for all distinct column-$f$-equivalent indices $k$ and $l$.
\end{lemma}

\begin{proof}
Let $k$ and $l$ be distinct column-$f$-equivalent indices.
We can choose a permutation $\sigma \in \mathfrak{S}_n$ such that $\sigma(k)=i$ and $\sigma(l)=j$.
Since $f$ is column-normalized and the indices $k$ and $l$ are column-$f$-equivalent, we have
$$f(\tau_{i,j} \sigma \tau_{k,l})=-f(\tau_{i,j} \sigma).$$
Besides, our assumptions on $i$ and $j$ show that
$$f(\tau_{i,j} \sigma \tau_{k,l})=-\frac{z_k}{z_l} f(\sigma \tau_{k,l}).$$
Finally, the choice of $\sigma$ shows that $\sigma \tau_{k,l} \sigma^{-1}=\tau_{\sigma(k),\sigma(l)}=\tau_{i,j}$, whence
$\sigma \tau_{k,l}=\tau_{i,j} \sigma$. Since $f(\sigma \tau_{k,l}) \neq 0$, we conclude that $z_l=z_k$, as claimed.
\end{proof}

\begin{lemma}\label{columnnormalizationLemma}
Every matrix functional is H-equivalent to a column-normalized one.
\end{lemma}

\begin{proof}
For every column-$f$-equivalence class $\calO$, we denote by $n_\calO$ its least element.
We define $L$ as the set of all $n_\calO$ where $\calO$ ranges over the column-$f$-equivalence classes.
Let $k \in \lcro 1,n\rcro \setminus L$, whose column-$f$-equivalence class we denote by $\calO$.
Then, there is a vector $Z^{(k)}=(z_i^{(k)})_{1 \leq i \leq n}$ with entries in $\F^*$ such that $\widetilde{f}(M)=0$ for all $M \in \Mat_n(\F)$
such that $C_k(M)=Z^{(k)}\star C_{n_\calO}(M)$. We set
$$a_{i,j}:=\begin{cases}
1 & \text{if $j \in L$} \\
z_i^{(j)} & \text{otherwise,}
\end{cases}$$
thereby defining a matrix $A \in \Mat_n(\F^*)$.
Then, $M \in \Mat_n(\F) \mapsto \widetilde{f}(A \star M)$ is the Schur functional attached to some mapping $g : \mathfrak{S}_n \rightarrow \F^*$, and $g$ is H-equivalent to $f$.
Moreover, it follows from Lemma \ref{Hequivalencelemma} that for every column-$f$-equivalence class $\calO$ and every $j \in \calO \setminus \{n_\calO\}$,
we have $\widetilde{g}(M)=0$ for all $M \in \Mat_n(\F)$ such that $C_j(M)=C_{n_\calO}(M)$.
Hence, by Lemma \ref{Transitivitylemma}, for all distinct column-$f$-equivalent indices $i,j$, we have $\widetilde{g}(M)=0$ for all $M \in \Mat_n(\F)$ such that $C_j(M)=C_i(M)$. By Corollary \ref{HequivalenceCorollary}, we know that column-$g$-equivalence is column-$f$-equivalence, and we conclude that $g$ is column-normalized.
\end{proof}

\begin{proof}[Proof of Proposition \ref{normalizationProp}]
By Lemma \ref{columnnormalizationLemma}, we lose no generality in assuming that $f$ is already column-normalized.

Let $\calO$ be a row-$f$-equivalence class, the least element of which we denote by $n_\calO$.
For all $i \in \calO \setminus \{n_\calO \}$, we have a vector $Y^{(i)}=(y_{i,j})_{1 \leq j \leq n} \in (\F^*)^n$ such that $\widetilde{f}(M)=0$ for all $M \in \Mat_n(\F)$ satisfying $R_i(M)=(Y^{(i)})^T \star R_{n_\calO}(M)$.
Denote by $L$ the set of all integers $n_\calO$ where $\calO$ ranges over the set of all row-$f$-equivalence classes.
For $(i,j)\in \lcro 1,n\rcro^2$, set
$$b_{i,j}:=\begin{cases}
1 & \text{if $i \in L$} \\
y_{i,j} & \text{otherwise,}
\end{cases}$$
thereby defining a matrix $B \in \Mat_n(\F^*)$.
With the same line of reasoning as in the proof of Lemma \ref{columnnormalizationLemma} (applied to $f^T$), we obtain
that the Schur functional $\widetilde{g} : M \mapsto \widetilde{f}(B \star M)$ is row-normalized. To complete the proof, we will show that $g$ is also column-normalized.
Indeed, since $f$ is column-normalized, it follows from Lemma \ref{forcingrowlemma} that, for every row-$f$-equivalence
class $\calO$ and all $i \in \calO \setminus \{n_\calO\}$, we have $y_{i,j}=y_{i,k}$ for all
distinct column-$f$-equivalent indices $j$ and $k$. Hence, for all $i,j,k$ in $\lcro 1,n\rcro$, it follows that
$b_{i,j}=b_{i,k}$ whenever $j$ and $k$ are column-$f$-equivalent.
Since $f$ is column-normalized, it follows from Lemma \ref{Hequivalencelemma} that, for all distinct column-$f$-equivalent indices $j$ and $k$,
one has $\widetilde{g}(M)=0$ for all $M \in \Mat_n(\F)$ such that $C_j(M)=C_k(M)$.
As column-$g$-equivalence coincides with column-$f$-equivalence, we conclude that $g$ is column-normalized, and hence $g$
is normalized.
\end{proof}

\subsection{The reduction to fully-normalized functionals}\label{fullnormalizationSection}

Let $\tau$ and $\tau'$ be permutations of $\lcro 1,n\rcro$.
Let us consider the Schur functional $\widetilde{g} : M \mapsto \widetilde{f}(P_{\tau'} M P_\tau)$.
It is then easily checked that:
\begin{itemize}
\item Two indices $i,j$ in $\lcro 1,n\rcro$ are column-$f$-equivalent if and only if
$\tau(i)$ and $\tau(j)$ are column-$g$-equivalent.
\item Two indices $i,j$ in $\lcro 1,n\rcro$ are row-$f$-equivalent if and only if
$(\tau')^{-1}(i)$ and $(\tau')^{-1}(j)$ are row-$g$-equivalent.
\end{itemize}
Moreover, if $f$ is normalized then so is $g$.

\begin{Def}
Given a partition $\calF$ of a finite set $X$, we denote by $n(\calF)$ the list of
cardinalities of the elements of $\calF$, in non-increasing order. We say that $n(\calF)$
is the \textbf{cardinality list} of $\calF$.
\end{Def}

Classically, given two partitions $\calF$ and $\calG$ of the same finite set $X$, the following conditions are equivalent:
\begin{enumerate}[(i)]
\item There exists a permutation $\sigma$ of $X$ such that $\calG=\{\sigma(Y) \mid Y \in \calF\}$.
\item One has $n(\calF)=n(\calG)$.
\end{enumerate}

\begin{Not}
To the mapping $f$, we associate the cardinality list $c(f)$ of the set of all column-$f$-equivalence classes,
and the cardinality list $r(f)$ of the set of all row-$f$-equivalence classes.
\end{Not}

Write $c(f)=(p_1,\dots,p_b)$ and $r(f)=(n_1,\dots,n_a)$.
By the above remark, there are permutations $\tau$ and $\tau'$ of $\lcro 1,n\rcro$ such that
$$\Bigl\{\tau(\calO) \mid \calO \in \lcro 1,n\rcro/\underset{C,f}{\sim}\Bigr\}=
\Bigl\{\lcro 1,p_1\rcro, \; \lcro p_1+1,p_1+p_2\rcro, \dots,\lcro p_1+\cdots+p_{b-1}+1,p_1+\cdots+p_b\rcro\Bigr\}$$
and
$$\Bigl\{(\tau')^{-1}(\calO) \mid \calO \in \lcro 1,n\rcro/\underset{R,f}{\sim}\Bigr\}=
\Bigl\{\lcro 1,n_1\rcro, \; \lcro n_1+1,n_1+n_2\rcro, \dots,\lcro n_1+\cdots+n_{a-1}+1,n_1+\cdots+n_a\rcro\Bigr\}.$$
If $f$ is normalized, we deduce that the Schur functional $M \mapsto \widetilde{f}(P_{\tau'} M P_\tau)$ is fully-normalized.

Hence, by combining the previous study with Proposition \ref{normalizationProp}, we conclude:

\begin{prop}\label{fullnormalizedprop}
Every matrix functional is PH-equivalent to a fully-normalized one.
\end{prop}

\subsection{The case of central mappings}

A mapping $f : \mathfrak{S}_n \rightarrow \F^*$ is \textbf{central} whenever
it is constant on each conjugacy class in $\mathfrak{S}_n$.
Here, we shall establish the following result:

\begin{theo}\label{centralequivalencetheo}
Assume that $n \geq 3$, and let $f : \mathfrak{S}_n \rightarrow \F^*$ be a central mapping.
Then, exactly one of the following conditions holds:
\begin{enumerate}[(a)]
\item The mapping $f$ is rigid.
\item There are non-zero scalars $\alpha$ and $\beta$ such that
$f$ maps every $\sigma \in \mathfrak{S}_n$ to $\alpha \beta^{\nfix(\sigma)} \sgn(\sigma)$, where
$$\nfix(\sigma):=\big|\{i \in \lcro 1,n\rcro : \; \sigma(i)=i\}\big|$$
denotes the number of fixed points of $\sigma$.
\end{enumerate}
Moreover, in the second case $f$ is H-equivalent to the signature.
\end{theo}

The result fails when $n=2$ (in that case (a) does not hold, and (b) holds if and only if $-\frac{f(\tau_{1,2})}{f(\id)}$ is a square in $\F$, which might fail).

\begin{proof}[Proof of Theorem \ref{centralequivalencetheo}]
To start with, we prove that if condition (b) holds then $f$ is H-equivalent to the signature, and hence condition (a) fails.
Assume indeed that there are non-zero scalars $\alpha$ and $\beta$ such that $f : \sigma \mapsto \alpha \beta^{\nfix(\sigma)} \sgn(\sigma)$.
Define $A=(a_{i,j})\in \Mat_n(\F^*)$ by $a_{i,j}:=\beta^{-1}$ if $i \neq j$, and $a_{i,j}:=1$ otherwise.
Define $B=(b_{i,j})\in \Mat_n(\F^*)$ by $b_{i,j}=\alpha \beta^n$ if $j=1$, and $b_{i,j}=1$ otherwise.
One computes that
 $$\forall M \in \Mat_n(\F), \; \det(A \star M)=\sum_{\sigma \in \mathfrak{S}_n} \sgn(\sigma) \beta^{\nfix(\sigma)-n} \prod_{j=1}^n m_{\sigma(j),j}
 =\beta^{-n} \alpha^{-1} \widetilde{f}(M)$$
 and hence
 $$\forall M \in \Mat_n(\F), \; \det\bigl((B \star A) \star M\bigr)=\det\bigl(B\star (A \star M)\bigr)=\alpha \beta^n \det(A \star M)=\widetilde{f}(M).$$

Now, we seek to prove that (a) or (b) holds.

As the situation is unchanged in multiplying $f$ with a non-zero scalar, we lose no generality in assuming that $f(\id)=1$.

For all $\sigma \in \mathfrak{S}_n$, we know that $\sigma^{-1}$ is conjugated to $\sigma$ in $\mathfrak{S}_n$,
whence $f(\sigma^{-1})=f(\sigma)$. It follows that $f=f^T$, and we deduce that row-$f$-equivalence coincides with column-$f$-equivalence.

Assume now that condition (a) fails. Then, we can find distinct indices $i,j$ in $\lcro 1,n\rcro$ that are column-$f$-equivalent.
Let $\sigma \in \mathfrak{S}_n$.
Since $f$ is central, we see that $\widetilde{f}(P_\sigma^{-1} M P_\sigma)=\widetilde{f}(M)$ for all $M \in \Mat_n(\F)$.
It follows from the remarks at the start of Section \ref{fullnormalizationSection} that $\sigma(i)$ and $\sigma(j)$ are column-$f$-equivalent.
Hence, by varying $\sigma$ we deduce that any two (distinct) indices in $\lcro 1,n\rcro$ are column-$f$-equivalent.

Next, we can choose a vector $Z=(z_k) \in (\F^*)^n$ such that
$z_2 =1$ and
$$\forall \sigma \in \mathfrak{S}_n, \; f(\sigma \tau_{1,2})=-\frac{z_{\sigma(2)}}{z_{\sigma(1)}} f(\sigma).$$
Let $\sigma' \in \mathfrak{S}_n$ be such that $\sigma'(1)=1$.

For all $\sigma \in \mathfrak{S}_n$, we get from the centrality of $f$ that
$$f\bigl(\sigma'(\sigma \tau_{1,2})(\sigma')^{-1}\bigr)=f(\sigma  \tau_{1,2})=
-\frac{z_{\sigma(2)}}{z_{\sigma(1)}} f(\sigma)=-\frac{z_{\sigma(2)}}{z_{\sigma(1)}}f\bigl(\sigma' \sigma (\sigma')^{-1}\bigr),$$
that is
$$f\bigl((\sigma' \sigma (\sigma')^{-1}) \tau_{\sigma'(1),\sigma'(2)}\bigr)
=-\frac{z_{\sigma(2)}}{z_{\sigma(1)}}f\bigl(\sigma' \sigma (\sigma')^{-1}\bigr).$$
It follows that the vector
$Z':=\bigl(z_{(\sigma')^{-1}(i)}\bigr)_{1 \leq i \leq n}$ satisfies :
\begin{equation}\label{Zeq}
\forall \sigma \in \mathfrak{S}_n, \; f(\sigma \tau_{1,\sigma'(2)})=
-\frac{z'_{\sigma(\sigma'(2))}}{z'_{\sigma(1)}}\,f(\sigma).
\end{equation}

Assume now that $\sigma'(2)=2$. Then, we get that $Z'$ is collinear with $Z$.
As $z'_1=z_1$, we deduce that $Z'=Z$. Varying $\sigma'$, we deduce that
$i \mapsto z_i$ is constant on $\lcro 3,n\rcro$. In the remainder of the proof, we set
$\mu:=z_3$ and $\lambda:=z_1$.

Now, let $k \in \lcro 2,n\rcro$. Then, we define $Z^{(k)}=(z_i^{(k)})_{1 \leq i \leq n}$ by
$$z_i^{(k)}=\begin{cases}
\lambda & \text{if $i=1$} \\
1 & \text{if $i=k$} \\
\mu & \text{otherwise}
\end{cases}$$
and it follows from \eqref{Zeq} that $\widetilde{f}(M)=0$ for all $M \in \Mat_n(\F)$
such that $C_k(M)=Z^{(k)} \star C_1(M)$.
Define then $A$ as the matrix whose columns are $E,Z^{(2)},\dots,Z^{(n)}$,
where $E$ is the vector $(1)_{1 \leq i \leq n}$ of $\F^n$.
By coming back to the proof of Lemma \ref{columnnormalizationLemma}, we obtain that
the mapping $g : \mathfrak{S}_n \rightarrow \F^*$ that is associated with the
Schur functional $M \mapsto \widetilde{f}(A \star M)$ is column-normalized. By Remark \ref{detremark}, it follows that there exists a non-zero scalar $\nu$ such that $\widetilde{f}(M)=\nu \det(A^{[-1]} \star M)$ for all $M \in \Mat_n(\F)$. Since $f(\id)=1$ and all the diagonal entries of $A$
equal $1$, we actually have $\nu=1$, whence $\widetilde{f}(M)=\det(A^{[-1]} \star M)$ for all $M \in \Mat_n(\F)$.

By considering $\widetilde{f}(P_{\tau_{2,3}})$ and $\widetilde{f}(P_{\tau_{1,3}})$, we find $f(\tau_{2,3})=- \mu^{-2}$
and $f(\tau_{1,3})=-\lambda^{-1}$. Since $f$ is central, it follows that
$\lambda=\mu^2$.
Let then $X \in \F^n$ and $Y \in \F^n$ be defined by $x_1=y_1=1$ and $x_i=\mu^{-1}$ and $y_i=\mu$ for all $i \geq 2$. We know from Lemma \ref{HadProductLemma} that $\det((XY^T)\star M)=\det(M)$ for all $M \in \Mat_n(\F)$. Hence, with $A':=(XY^T) \star A^{[-1]}$, we deduce that
$$\forall M \in \Mat_n(\F), \; \widetilde{f}(M)=\det(A' \star M).$$
Finally, one sees that $A'$ is the matrix whose diagonal entries all equal $1$, and whose off-diagonal entries all equal $\mu^{-1}$.
It follows that
$$\forall \sigma \in \mathfrak{S}_n, \; f(\sigma)=\det(A' \star P_\sigma)=\mu^{\nfix(\sigma)-n}\sgn(\sigma),$$
which validates condition (b) for $\alpha:=\mu^{-n}$ and $\beta:=\mu$.
\end{proof}


\subsection{Further results on column and row equivalence}\label{possibleclassesSection}

In the previous paragraphs, we have seen that there are some restrictions on the possible column-$f$-equivalence classes, and ditto
for row-$f$-equivalence: if column-$f$-equivalence relates all the indices in $\lcro 1,n\rcro$ then so does row-$f$-equivalence
(see Remark \ref{detremark}); moreover, if $n=2$ then any two elements of $\{1,2\}$ are column-$f$-equivalent and row-$f$-equivalent.

First, we generalize the latter result as follows:

\begin{prop}\label{notn-1}
Let $f : \mathfrak{S}_n \rightarrow \F^*$. Then, no column-$f$-equivalence class has cardinality $n-1$, and no
row-$f$-equivalence class has cardinality $n-1$.
\end{prop}

\begin{proof}
It suffices to consider column-$f$-equivalence. Without loss of generality,
we can assume that any two elements of $\lcro 1,n-1\rcro$ are column-$f$-equivalent, and we aim at proving that $1$ is column-$f$-equivalent to $n$. Without loss of generality, we can assume further that $f$ is column-normalized. Then, it suffices to prove that $f$ is H-equivalent to the signature.

To start with, we know that $f(\sigma \tau_{i,j})=-f(\sigma)$ for all distinct $i,j$ in $\lcro 1,n-1\rcro$ and all $\sigma \in \mathfrak{S}_n$.
Since $\mathfrak{S}_{n-1}$ is generated by transpositions, it follows that $f(\sigma  \tau)=\sgn(\tau) f(\sigma)$
for all $(\sigma,\tau)\in (\mathfrak{S}_n)^2$ such that $\tau(n)=n$.
Hence, $f(\sigma)=\sgn\bigl((\sigma')^{-1} \sigma\bigr)f(\sigma')$ for all $(\sigma,\sigma')\in (\mathfrak{S}_n)^2$ such that $\sigma(n)=\sigma'(n)$.

For $k \in \lcro 1,n-1\rcro$, set $b_k:=-f(\tau_{k,n})$. Set also $b_n:=f(\id)$.
It then follows from the above results that
$$\forall \sigma \in \mathfrak{S}_n, \; f(\sigma)=\sgn(\sigma)\, b_{\sigma(n).}$$
Defining $A:=(a_{i,j}) \in \Mat_n(\F)$ by $a_{i,j}:=1$ if $j<n$, and $a_{i,j}:=b_i$ if $j=n$,
we conclude that $f$ is H-equivalent to $\sgn$. Hence, any two indices in $\lcro 1,n\rcro$ are column-$f$-equivalent.
\end{proof}

Here is another phenomenon that is quite specific to the case $n=4$.

\begin{prop}\label{fourbyfourprop}
Let $f : \mathfrak{S}_4 \rightarrow \F^*$. Assume that the column-$f$-equivalence classes are $\lcro 1,2\rcro$ and $\lcro 3,4\rcro$ and that $1$ and $2$ are row-$f$-equivalent. Then, $3$ and $4$ are row-$f$-equivalent.
\end{prop}

\begin{proof}
Without loss of generality, we can assume that $f$ is normalized.
Consider the subgroup $G$ of $\mathfrak{S}_4$ generated by $\tau_{1,2}$ and $\tau_{3,4}$.
The assumptions show that $f(\sigma \tau)=\sgn(\tau) f(\sigma)$ for all $\sigma \in \mathfrak{S}_4$
and $\tau \in G$. Clearly, two permutations in $\mathfrak{S}_4$ belong to the same orbit under
the right-action of $G$ by right-multiplication if and only if they map $\{1,2\}$ to the same subset.
Moreover, we have $f(\tau_{1,2} \sigma)=-f(\sigma)$ for all $\sigma \in \mathfrak{S}_4$.
Under the action of the subgroup $\{\id,\tau_{1,2}\}$, there are four orbits of subsets of $\{1,2,3,4\}$
with cardinality $2$: the singletons $\{\{1,2\}\}$ and $\{\{3,4\}\}$ and the pairs $\{\{1,3\},\{2,3\}\}$
and $\{\{1,4\},\{2,4\}\}$. Hence, there are nonzero scalars $x,y,u,t$ such that
$$\forall \sigma \in \mathfrak{S}_4, \; f(\sigma)=\begin{cases}
\sgn(\sigma)\,x & \text{if $\sigma(\{1,2\})=\{1,2\}$} \\
\sgn(\sigma)\,y & \text{if $\sigma(\{1,2\})=\{3,4\}$} \\
\sgn(\sigma)\,u & \text{if $\sigma(\{1,2\}) \cap \{3,4\}=\{3\}$} \\
\sgn(\sigma)\,t & \text{if $\sigma(\{1,2\}) \cap \{3,4\}=\{4\}$.}
\end{cases}$$
Set $Z:=\begin{bmatrix}
u & u & t & t
\end{bmatrix}^T=(z_k)_{1 \leq k \leq 4}$. For all $\sigma \in \mathfrak{S}_4$, if $\sigma(\{1,2\})=\{1,2\}$ or $\sigma(\{1,2\})=\{3,4\}$
then $(\tau_{3,4} \sigma)(\{1,2\})=\sigma(\{1,2\})$, and we see that $\sigma^{-1}(3)$ and $\sigma^{-1}(4)$
belong both to $\{1,2\}$ or both to $\{3,4\}$, whence $z_{\sigma^{-1}(3)}=z_{\sigma^{-1}(4)}$.

Let $\sigma \in \mathfrak{S}_4$ be such that $\sigma(\{1,2\}) \cap \{3,4\}=\{3\}$.
Then, $(\tau_{3,4} \sigma)(\{1,2\}) \cap \{3,4\}=\{4\}$.
Moreover, $\{1,2\} \cap \{\sigma^{-1}(3),\sigma^{-1}(4)\}=\{\sigma^{-1}(3)\}$, and hence
$\sigma^{-1}(3) \in \{1,2\}$ and $\sigma^{-1}(4) \in \{3,4\}$.
It follows that
$$f(\tau_{3,4} \sigma)=-\frac{t}{u} f(\sigma)=-\frac{z_{\sigma^{-1}(4)}}{z_{\sigma^{-1}(3)}}\, f(\sigma).$$
Likewise, one shows that for all $\sigma \in \mathfrak{S}_4$ such that $\sigma(\{1,2\}) \cap \{3,4\}=\{4\}$,
$$f(\tau_{3,4} \sigma)=-\frac{u}{t}\, f(\sigma)=-\frac{z_{\sigma^{-1}(4)}}{z_{\sigma^{-1}(3)}}\, f(\sigma).$$
Hence, we have shown that $f(\tau_{3,4} \sigma)=-\frac{z_{\sigma^{-1}(4)}}{z_{\sigma^{-1}(3)}}\, f(\sigma)$
for all $\sigma \in \mathfrak{S}_4$, and we conclude that $3$ is row-$f$-equivalent to $4$.
\end{proof}

Now, let us give examples of column and row equivalence classes. First, for $n=4$
we have an example that is related to the previous result.

\begin{ex}
Assume that the field $\F$ has more than $2$ elements. Then, we choose $x \in \F \setminus \{0,1\}$ and
we define a mapping $f : \mathfrak{S}_4 \rightarrow \F^*$ as follows:
$$f : \sigma \in \mathfrak{S}_4 \mapsto \begin{cases}
\sgn(\sigma)\,x & \text{if $\{\sigma(1),\sigma(2)\} = \{1,2\}$} \\
\sgn(\sigma) & \text{otherwise.}
\end{cases}$$
For all $\sigma \in \mathfrak{S}_4$, note that
$(\sigma \tau_{1,2})(\{1,2\})=\sigma(\{1,2\})=(\sigma \tau_{3,4})(\{1,2\})$, hence
$f(\sigma \tau_{1,2})=-f(\sigma)=f(\sigma\tau_{3,4})$
and it follows that $1 \underset{C,f}{\sim} 2$ and $3 \underset{C,f}{\sim} 4$. We claim however that $1$ is not column-$f$-equivalent to $4$.
Indeed, if the contrary held then the ratio $\frac{f(\sigma \tau_{1,4})}{f(\sigma)}$ would depend only on the pair $(\sigma(1),\sigma(4))$.
Yet, with $\sigma=\id$ this ratio equals $-x^{-1}$, whereas with $\sigma=\tau_{2,3}$ it equals $-1$.
Hence, the column-$f$-equivalence classes are $\{1,2\}$ and $\{3,4\}$.

Here, one sees that $f(\sigma^{-1})=f(\sigma)$ for all $\sigma \in \mathfrak{S}_4$, and hence the row-$f$-equivalence classes
are also $\{1,2\}$ and $\{3,4\}$.
\end{ex}

\begin{ex}
Let $n \geq 5$, and assume that $\F$ has at least three elements. Choose $x$ in $\F \setminus \{0,1\}$.
We define
$$g : \sigma\in \mathfrak{S}_4 \mapsto \begin{cases}
\sgn(\sigma)\,x & \text{if $\{\sigma(1),\sigma(2)\} = \{2,n\}$ or $\{\sigma(1),\sigma(2)\} = \{1,n\}$} \\
\sgn(\sigma) & \text{otherwise.}
\end{cases}$$
As in the previous example, one proves that $1 \underset{C,g}{\sim} 2$ and $i \underset{C,g}{\sim} j$
for all distinct $i,j$ in $\lcro 3,n\rcro$. Moreover, for $\sigma:=\tau_{1,n}$,
the ratio $\frac{g(\sigma \tau_{1,3})}{g(\sigma)}$ equals $-x^{-1}$, whereas for any $\sigma$ such that $\sigma(1)=n$, $\sigma(2)=n-1$
and $\sigma(3)=3$, this ratio equals $-1$. Hence, $1$ is not column-$g$-equivalent to $n$, and we deduce
from Proposition \ref{notn-1} that the column-$g$-equivalence classes are $\{1,2\}$ and $\lcro 3,n\rcro$. Note also that $g$ is column-normalized.

We remark that the condition that $\{\sigma(1),\sigma(2)\} = \{2,n\}$ or $\{\sigma(1),\sigma(2)\} = \{1,n\}$
is invariant in replacing $\sigma$ with $\tau \sigma$ for some permutation $\tau$ that fixes $n$
and leaves $\{2,1\}$ invariant. Hence, $1$ and $2$ are row-$g$-equivalent and $3,\dots,n-1$ are row-$g$-equivalent.
Let us prove that $1$ and $n$ are not row-$g$-equivalent. Assume the contrary.
Then, there is a vector $(z_k)\in (\F^*)^n$ such that $g(\tau_{1,n} \sigma)=-\frac{z_{\sigma^{-1}(n)}}{z_{\sigma^{-1}(1)}}
g(\sigma)$ for all $\sigma \in \mathfrak{S}_n$. By Lemma \ref{forcingrowlemma}, we have
$z_1=z_2$ and $z_3=\cdots=z_n$. Take $\sigma \in \mathfrak{S}_n$ such that $\sigma(1)=1$, $\sigma(2)=3$ and $\sigma(3)=n$.
Then, $g(\sigma)=\sgn(\sigma)$ and $g(\tau_{1,n}\sigma)=-\sgn(\sigma)$, leading to $z_1=z_3$. Hence
$g(\tau_{1,n}\sigma)=-g(\sigma)$ for all $\sigma \in \mathfrak{S}_n$. However with $\sigma=\id$
this leads to $x=1$, a contradiction.

Next, we prove that $3$ and $n$ are not row-$g$-equivalent. Assume otherwise, so that we have a vector
$Z'=(z'_k)\in (\F^*)^n$ such that $g(\tau_{3,n} \sigma)=-\frac{z'_{\sigma^{-1}(n)}}{z'_{\sigma^{-1}(3)}}
g(\sigma)$ for all $\sigma \in \mathfrak{S}_n$.
Again, $z'_1=z'_2$ and $z'_3=\cdots=z'_n$. As $n \geq 5$ we can choose $\sigma \in \mathfrak{S}_n$ such that
$\sigma(3)=n$, $\sigma(1)=3$ and $\sigma(2)=n-1$. Then, $g(\sigma)=\sgn(\sigma)$ and $g(\tau_{3,n}\sigma)=-\sgn(\sigma)$,
leading to $z'_3=z'_1$. Hence, $g(\tau_{3,n}\sigma)=-g(\sigma)$ for all $\sigma \in \mathfrak{S}_n$.
Taking $\sigma:=\tau_{1,3}$, we have $g(\sigma)=-1$ and $g(\tau_{3,n}\sigma)=x$, and hence $x=1$. Again, this is a contradiction.

Finally, if some element of $\{1,2\}$ were row-$g$-equivalent to some element of $\lcro 3,n-1\rcro$,
then Proposition \ref{notn-1} would yield that all the indices in $\lcro 1,n\rcro$ are row-$g$-equivalent,
which has just been disproved. We conclude that the row-$g$-equivalence classes are $\{1,2\},\lcro 3,n-1\rcro$ and $\{n\}$.
\end{ex}

\begin{ex}
Let $n \geq 4$, and assume that $\F \setminus \{0\}$ contains at least $n-1$ elements $x_1,\dots,x_{n-1}$.
We define
$$h :\sigma \in \mathfrak{S}_n \mapsto \begin{cases}
\sgn(\sigma)\,x_i & \text{if $\{\sigma(1),\sigma(2)\} = \{i,n\}$ for some $i \in \lcro 1,n-1\rcro$} \\
\sgn(\sigma) & \text{otherwise.}
\end{cases}$$
One sees that $h(\sigma \tau_{1,2})=-h(\sigma)$ for all $\sigma \in \mathfrak{S}_n$, and
 $h(\sigma \tau_{i,j})=-h(\sigma)$ for all distinct indices $i,j$ in $\lcro 3,n\rcro$ and all $\sigma \in \mathfrak{S}_n$.
Hence, $1$ and $2$ are column-$h$-equivalent, and $3,\dots,n$ are all column-$h$-equivalent.
Again, let us prove that $1$ is not column-$h$-equivalent to $n$. If the contrary held,
then the ratio $\frac{h(\sigma \tau_{1,n})}{h(\sigma)}$ would depend only on the pair $(\sigma(1),\sigma(n))$,
which is contradicted by taking the permutations $\id$ (for which the ratio equals $-x_2$) and $\tau_{2,3}$
(for which the ratio equals $-x_3$). Hence, the column-$h$-equivalence classes are $\{1,2\}$ and $\lcro 3,n\rcro$.
Moreover, $h$ is column-normalized.

Let us prove that no two distinct indices in $\lcro 1,n\rcro$ are row-$h$-equivalent.
Let $i,j$ be distinct indices that are row-$h$-equivalent. Then,
there is a vector $(z_k) \in (\F^*)^n$ such that $h(\tau_{i,j}\sigma)=-\frac{z_{\sigma^{-1}(j)}}{z_{\sigma^{-1}(i)}}h(\sigma)$
for all $\sigma \in \mathfrak{S}_n$. By Lemma \ref{forcingrowlemma}, we already know that
$z_1=z_2$ and $z_3=\cdots=z_n$.
\begin{itemize}
\item Assume first that $i=1$ and $j=2$.
We can choose a permutation $\sigma$ such that $\sigma(1)=1$, $\sigma(2)=3$ and $\sigma(3)=2$.
Then, $\tau_{1,2} \sigma$ maps $\{1,2\}$ to $\{2,3\}$, which does not contain $n$.
It follows from the definition of $h$ that $h(\tau_{1,2} \sigma)=-\sgn(\sigma)=-h(\sigma)$, and we deduce that
$z_3=z_1$. Hence, $h(\tau_{1,2}\sigma)=-h(\sigma)$ for all $\sigma \in \mathfrak{S}_n$.
Taking $\sigma=\tau_{1,n}$, and we see that $h(\sigma)=-x_2$ while $h(\tau_{1,2} \sigma)=x_1$, a contradiction.

\item Assume that $i=1$ and $j=n$.
Choose $\sigma \in \mathfrak{S}_n$ such that $\sigma(1)=1$, $\sigma(2)=2$ and $\sigma(3)=n$. 
Then, $h(\sigma)=\sgn(\sigma)$ and $h(\tau_{1,n}\sigma)=-\sgn(\sigma)\, x_2$, and hence 
$z_3=x_2 z_1$. 

Choose $\sigma \in \mathfrak{S}_n$ such that $\sigma(1)=1$, $\sigma(2)=3$ and $\sigma(4)=n$.
Then, $h(\sigma)=\sgn(\sigma)$ and $h(\tau_{1,n}\sigma)=-\sgn(\sigma)\, x_3$, and hence
$z_4=x_3 z_1$. As $z_4=z_3$, we conclude that $x_2=x_3$, which contradicts our assumptions. 
\end{itemize}
We deduce that $1$ is neither row-$g$-equivalent to $2$ nor to $n$. Symmetrically, no two distinct elements of $\lcro 1,n-1\rcro$
are row-$h$-equivalent, and $n$ is row-$h$-equivalent to no element of $\lcro 1,n-1\rcro$. Hence, no
two distinct elements of $\lcro 1,n\rcro$ are row-$h$-equivalent.
\end{ex}

In this example, if in the definition of $h$ we replace the condition
$\{\sigma(1),\sigma(2)\}=\{i,n\}$ by $(\sigma(1),\sigma(2))=(n,i)$ then one can show that
the resulting mapping has its column partition equal to $\bigl\{\{1\},\{2\},\lcro 3,n\rcro\bigr\}$
and that its row partition is the set of all singletons of $\lcro 1,n\rcro$.

At this point, we have the following conjecture:

\begin{conj}
Let $\calF$ be a partition of $\lcro 1,n\rcro$, and $\F$ be an infinite field.
Assume that $\calF$ does not have exactly two elements, one of which is a singleton.
Then, there exists a mapping $f : \mathfrak{S}_n \rightarrow \F^*$
such that  $\calF$ is the quotient set for column-$f$-equivalence.
\end{conj}

Here is an even more challenging open problem: given an infinite field $\F$,
describe the pairs $(\calF,\calG)$ of partitions of $\lcro 1,n\rcro$ for which there exists a mapping
$f : \mathfrak{S}_n \rightarrow \F^*$ whose column-equivalence classes are the elements of $\calF$
and whose row-equivalence classes are the elements of $\calG$.

\section{Main results}

Now that we have defined the column and row equivalence relations attached to a mapping $f$,
we can state some of our results on the structure of $(f,g)$-transformations.
Remembering the definition of the row list $r(f)$ and the column list $c(f)$ of $f$,
we will prove:

\begin{theo}\label{maintheo}
Let $f$ and $g$ be mappings from $\mathfrak{S}_n$ to $\F^*$, and $U$ be an endomorphism of the vector space
$\Mat_n(\F)$ such that
$$\forall M \in \Mat_n(\F), \; \widetilde{g}(U(M))=0 \Leftrightarrow \widetilde{f}(M)=0.$$
Assume that $n \geq 2$.
Then:
\begin{enumerate}[(i)]
\item $U$ is bijective.
\item The mapping $g$ is PH-equivalent to $f$ or to $f^T$.
\item We have $(r(g),c(g))=(r(f),c(f))$ or $(r(g),c(g))=(c(f),r(f))$.
\item There exists a non-zero scalar $\alpha$ such that $U$ is an $(\alpha f,g)$-transformation.
\end{enumerate}
\end{theo}

Remember that the column and row partitions of the signature of $\mathfrak{S}_n$ consist of the sole set $\lcro 1,n\rcro$, whereas the constant mapping equal to $1$ is rigid if $n>2$ and $\chi(\F) \neq 2$. Hence, as a special case of the above result (using point (iii) only), we get the following corollary:

\begin{cor}
Assume that $n \geq 3$ and $\car(\F) \neq 2$. Then:
\begin{enumerate}[(a)]
\item No endomorphism $U$ of $\Mat_n(\F)$ satisfies
$$\forall M \in \Mat_n(\F), \; \det \bigl(U(M)\bigr)=0 \Leftrightarrow \per M=0.$$
\item No endomorphism $U$ of $\Mat_n(\F)$ satisfies
$$\forall M \in \Mat_n(\F), \; \per \bigl(U(M)\bigr)=0 \Leftrightarrow \det M=0.$$
\end{enumerate}
\end{cor}

In particular, point (b) generalizes an earlier result of Duffner and da Cruz \cite{DuffnerdaCruz},
which was known only for fields with cardinality greater than or equal to $n$.

Theorem \ref{maintheo} will be proved over the course of the next two sections: in Section \ref{conesection},
we study the null cone of a Schur functional, and in Section \ref{transfosection} we give a
partial description for all $(f,g)$-transformations that is sufficiently precise so as to yield Theorem \ref{maintheo}.

\section{Vector spaces of matrices in the null cone of a matrix functional}\label{conesection}

Throughout this section, we fix a mapping $f : \mathfrak{S}_n \rightarrow \F^*$.

\begin{Def}
We define
$$\calC(f):=\bigl\{M \in \Mat_n(\F) : \; \widetilde{f}(M)=0\bigr\}=\widetilde{f}^{-1} \{0\},$$
which we call the \textbf{null cone} of $f$.
\end{Def}

Assume now that $f$ is normalized. A nonzero vector $(x_i) \in \F^n \setminus \{0\}$ is called \textbf{column-$f$-adapted} whenever
its \textbf{support}, defined as $\{i \in \lcro 1,n\rcro : \; x_i \neq 0\}$, is included in a column-$f$-equivalence class.
A nonzero vector $X\in \F^n \setminus \{0\}$ is called \textbf{row-$f$-adapted} whenever
its support is included in a row-$f$-equivalence class.

Given a non-zero vector $X \in \F^n \setminus \{0\}$, we set
$$\calV_X:=\bigl\{M \in \Mat_n(\F) : MX=0\bigr\} \quad \text{and} \quad \calV^X:=(\calV_X)^T=\bigl\{M \in \Mat_n(\F) : X^TM=0\bigr\}.$$
Both are linear subspaces of $\Mat_n(\F)$ with codimension $n$.

\begin{lemma}
Let $f : \mathfrak{S}_n \rightarrow \F^*$ be a normalized mapping. Let $X \in \F^n\setminus \{0\}$.
\begin{itemize}
\item If $X$ is column-$f$-adapted, then $\calV_X \subset \calC(f)$.
\item If $X$ is row-$f$-adapted, then $\calV^X \subset \calC(f)$.
\end{itemize}
\end{lemma}

\begin{proof}
Assume that $X$ is column-$f$-adapted.
Set $E:=\{i \in \lcro 1,n\rcro : x_i \neq 0\}$ and $a:=\min E$. Let $M \in \calV_X$. For
$i \in E \setminus \{a\}$, denote by $M_i$ the matrix whose columns are the same ones as for $M$, with the exception of the
$a$-th which equals $C_i(M)$; noting that $a \underset{C,f}{\sim} i$, we obtain that $\widetilde{f}(M_i)=0$ because $f$ is normalized.

Obviously, $\widetilde{f}$ is linear with respect to each column, and we have
$C_a(M)=-\underset{i \in E \setminus \{a\}}{\sum} \frac{x_i}{x_a}\,C_i(M)$
because $MX=0$. Hence,
$$\widetilde{f}(M)=-\underset{i \in E \setminus \{a\}}{\sum}\frac{x_i}{x_a}\,\widetilde{f}(M_i)=0.$$
This proves point (a). One proves point (b) in a similar way.
\end{proof}

In particular, by taking $X$ with exactly one non-zero entry, we find that $\calC(f)$ includes
linear subspaces with codimension $n$.
The main aim of the present section is to find a converse statement for the preceding lemma. This is done in two steps:

\begin{theo}\label{DimensionInequalityTheorem}
Let $\calV$ be an affine subspace of $\Mat_n(\F)$ that is included in $\calC(f)$.
Then, $\codim_{\Mat_n(\F)} \calV \geq n$.
\end{theo}

\begin{theo}\label{DimensionEqualityTheorem}
Let $f : \mathfrak{S}_n \rightarrow \F^*$ be a normalized function.
Let $V$ be a linear subspace of $\Mat_n(\F)$ that is included in $\calC(f)$ and has codimension $n$ in
$\Mat_n(\F)$.
Then, there exists a non-zero vector $X \in \F^n$ such that one of the following two situations holds:
\begin{enumerate}[(a)]
\item $V=\calV_X$ and $X$ is column-$f$-adapted;
\item $V=\calV^X$ and $X$ is row-$f$-adapted.
\end{enumerate}
\end{theo}

\subsection{A lemma}

The following basic lemma will be helpful to perform inductive proofs.

\begin{lemma}\label{upperblocklemma}
Let $\calV$ be an affine subspace of $\Mat_n(\F)$ that is included in $\calC(f)$.
Consider the subset $\calV'$ consisting of all matrices of $\calV$ of the form
$$M=\begin{bmatrix}
P(M) & [0]_{(n-1) \times 1} \\
[?]_{1 \times (n-1)} & 1
\end{bmatrix}$$
with $P(M) \in \Mat_{n-1}(\F)$.
Then, $P(\calV') \subset \calC(g)$ for some mapping $g : \mathfrak{S}_{n-1} \rightarrow \F^*$.
\end{lemma}

\begin{proof}
Every permutation $\sigma$ of $\lcro 1,n-1\rcro$ is naturally extended to a permutation $\overline{\sigma}$ of $\lcro 1,n\rcro$
such that $\overline{\sigma}(n)=n$. Then, with $g : \sigma \in \mathfrak{S}_{n-1} \mapsto f(\overline{\sigma})$,
one checks that
$$\forall M \in \calV', \; \widetilde{f}(M)=\widetilde{g}(P(M)),$$
and hence $P(\calV') \subset \calC(g)$.
\end{proof}

\subsection{Proof of Theorem \ref{DimensionInequalityTheorem}}

We prove the result by induction on $n$.
The case $n=1$ is obvious since the null cone of $f$ equals $\{0\}$ in that situation. \\
Assume now that $n \geq 2$. We perform a \emph{reductio ad absurdum} by assuming that
$\codim \calV<n$. Denote by $V$ the translation vector space of $\calV$.
If $\dim C_i(\calV) \leq n-1$ for all $i \in \lcro 1,n\rcro$, then
$$\dim \calV \leq \sum_{i=1}^n \dim C_i(\calV) \leq n(n-1),$$
contradicting our assumption that $\codim \calV<n$.

Hence, we can assume that $C_i(\calV)=\F^n$ for some $i \in \lcro 1,n\rcro$. By permuting columns
(which modifies the Schur functional we are working with), we see that no further generality
is lost in assuming that $C_n(\calV)=\F^n$.
Denote then by $\calV'$ the (non-empty) affine subspace of $\calV$ consisting of its matrices of the form
$$M=\begin{bmatrix}
P(M) & [0]_{(n-1) \times 1} \\
[?]_{1 \times (n-1)} & 1
\end{bmatrix} \quad \text{with $P(M) \in \Mat_{n-1}(\F)$.}$$
Note that $P(\calV')$ is an affine subspace of $\Mat_{n-1}(\F)$.
By Lemma \ref{upperblocklemma}, we obtain that $P(\calV')$ is included in the null cone of some mapping from $\mathfrak{S}_{n-1}$ to $\F^*$.
By induction, it follows that $\codim P(\calV') \geq n-1$.
On the other hand, by combining the rank theorem with the equality $C_n(\calV)=\F^n$, we find
$$n-1 \geq \codim \calV=\codim P(\calV')+((n-1)-\dim L),$$
where $L$ denotes the subspace of all matrices of $V$ in which the first $n-1$ rows and the last column equal zero.
It follows that $\dim L=n-1$, and hence $V$ contains $E_{n,j}$ for all $j \in \lcro 1,n-1\rcro$.
Using row permutations, we obtain likewise that $V$ contains $E_{i,j}$ for all $(i,j)\in \lcro 1,n\rcro \times \lcro 1,n-1\rcro$.
It follows that $P(\calV')$ contains $I_{n-1}$. Yet, for any $g : \mathfrak{S}_{n-1} \rightarrow \F^*$, we have
$\widetilde{g}(I_{n-1})=g(\id_{\lcro 1,n-1\rcro}) \neq 0$, which contradicts Lemma \ref{upperblocklemma}.

It follows that $\codim \calV \geq n$, and our inductive step is proved. Hence, Theorem \ref{DimensionInequalityTheorem}
is established.

\subsection{Two partial results on affine subspaces with the minimal codimension}

\begin{Not}
Given a subset $\calS$ of $\Mat_n(\F)$ and an index $i \in \lcro 1,n\rcro$, we denote:
\begin{itemize}
\item By $R'_i(\calS)$ the set of all matrices of $\calS$ in which all the rows are zero with the possible exception of the
$i$-th;
\item By $C'_i(\calS)$ the set of all matrices of $\calS$ in which all the columns are zero with the possible exception of the
$i$-th.
\end{itemize}
\end{Not}

Our starting point is the following lemma:

\begin{lemma}\label{Lemma1}
Assume that $n \geq 2$.
Let $\calV$ be an affine subspace of $\Mat_n(\F)$ with codimension $n$
that is included in the null cone of $f$.
Then, there exists an index $i \in \lcro 1,n\rcro$ such that $C_i(\calV)=\F^n$ or $R_i(\calV)=\Mat_{1,n}(\F)$.
\end{lemma}

\begin{proof}
Assume that the contrary holds. Denote by $V$ the translation vector space
of $\calV$, by $V^T$ its orthogonal complement for the standard symmetric bilinear form $(M,N)\mapsto \tr(MN)$,
and by $(e_1,\dots,e_n)$ the standard basis of $\F^n$. Then,
we know that, for all $i \in \lcro 1,n\rcro$, the space $C'_i(V^\bot)$ contains
a non-zero matrix $N_i$ (because $R_i(V) \neq \F^n$). The $N_i$ matrices are then linearly independent,
and as $\dim V^\bot=n$ we deduce that $V^\bot=\Vect(N_1,\dots,N_n)$.

Likewise, we find that, for all $i \in \lcro 1,n\rcro$, the space $R'_i(V^\bot)$ contains
a non-zero matrix $M_i$. Hence, for all $i \in \lcro 1,n\rcro$, the matrix $M_i$ is a linear combination of $N_1,\dots,N_n$, leading to
$$\im M_i \subset \sum_{k=1}^n \im N_k.$$
In turn, this successively leads to
$$\sum_{k=1}^n \im N_k=\F^n$$
and, since $\im N_k$ has dimension $1$ for all $k \in \lcro 1,n\rcro$, to
$$\underset{k=1}{\overset{n}{\bigoplus}} \im N_k=\F^n.$$
Denoting by $(E_1,\dots,E_n)$ the standard basis of $\F^n$, this yields a basis
$(X_1,\dots,X_n)$ of $\F^n$ such that $N_k=X_k E_k^T$ for all $k \in \lcro 1,n\rcro$.
The matrix $P$ whose columns are $X_1,\dots,X_n$ is invertible. For all
$(a_1,\dots,a_n)\in \F^n$, $\underset{k=1}{\overset{n}{\sum}} a_k N_k=PD$ where $D$ denotes the diagonal matrix with diagonal entries
$a_1,\dots,a_n$, and the rank of $\underset{k=1}{\overset{n}{\sum}} a_k N_k$ equals the number of indices $k$ such that $a_k \neq 0$.
It follows that every rank $1$ matrix of $\Vect(N_1,\dots,N_n)$ is a scalar multiple of some $N_i$.

In turn, this yields, for all $i \in \lcro 1,n\rcro$, a unique index $\sigma(i) \in \lcro 1,n\rcro$ such that $M_i \in \Vect(N_{\sigma(i)})$.
The map $\sigma : \lcro 1,n\rcro \rightarrow \lcro 1,n\rcro$ is obviously injective since the $M_i$'s are linearly independent.
Hence, $\sigma$ is a permutation of $\lcro 1,n\rcro$ and $E_{i,\sigma(i)} \in V^\bot$
for all $i \in \lcro 1,n\rcro$. Permuting columns (which changes the mapping $f$ we are working with),
we see that no generality is lost in assuming that $\sigma$ is the identity of $\lcro 1,n\rcro$.
In that case $V^\bot$ includes $\Vect(E_{i,i})_{1 \leq i \leq n}$, and as both spaces have dimension $n$
it follows that $V^\bot=\Vect(E_{i,i})_{1 \leq i \leq n}$.

In that reduced situation, we obtain fixed scalars $a_1,\dots,a_n$ such that $\calV$
is the (affine) space of all matrices $M=(m_{i,j})$ in $\Mat_n(\F)$ such that
$$\forall k \in \lcro 1,n\rcro, \; m_{k,k}=a_k.$$
In particular, $\calV$ contains the diagonal matrix with diagonal entries $a_1,\dots,a_n$, and
as this matrix must be annihilated by $\widetilde{f}$ some $a_k$ equals zero.
Permuting rows and columns, we are further reduced to the situation where $a_1=0$.
Then, we consider the matrix $M=(m_{i,j})$ of $\calV$ defined by
$$m_{i,j}=\begin{cases}
a_i & \text{if $i=j$} \\
1 & \text{if $i=j+1$ mod. $n$} \\
0 & \text{otherwise.}
\end{cases}$$
Using $n \geq 2$, one checks that $\widetilde{f}(M)=f(\sigma) \neq 0$,
where $\sigma$ denotes the $n$-cycle that takes $i$ to $i+1$ for all $i \in \lcro 1,n-1\rcro$.
This contradicts the assumption that $\calV \subset \calC(f)$.
\end{proof}

\begin{lemma}\label{Lemma2}
Let $\calV$ be an affine subspace of $\Mat_n(\F)$ with codimension $n$
that is included in the null cone of $f$. Denote by $V$ its translation vector space.
Then, there exists an index $i \in \lcro 1,n\rcro$ such that $R'_i(V)=\{0\}$ or $C'_i(V)=\{0\}$.
\end{lemma}

\begin{proof}
Here, the proof is done by induction on $n$, with a strategy that is globally similar to the one
of the proof of Theorem \ref{DimensionInequalityTheorem}. The case $n=1$ is trivial.
Assume that $n=2$ and that the result fails. Then, we note that
$$\dim V \geq \dim C'_1(V)+\dim C'_2(V) \quad  \text{and} \quad \dim V \geq \dim R'_1(V)+\dim R'_2(V)$$
and hence all the spaces $C'_1(V)$, $C'_2(V)$, $R'_1(V)$ and $R'_2(V)$ have dimension $1$.
With exactly the same line of reasoning as in the proof of Lemma \ref{Lemma1},
we deduce that $V=\Vect(E_{i,\sigma(i)})_{i \in \{1,2\}}$ for some permutation $\sigma$ of $\{1,2\}$. \\
Without loss of generality, we can assume that $\sigma=\id$. Then,
there are fixed scalars $\alpha_1$ and $\alpha_2$ such that $\calV$ is the set of all $2$ by $2$ matrices
with diagonal entries $\alpha_1$ and $\alpha_2$, and a contradiction is derived from there just like in the proof of
Lemma \ref{Lemma1}.

Assume from now on that $n \geq 3$ and that the result fails.

\vskip 2mm
\noindent \textbf{Step 1: There exists an index $i \in \lcro 1,n\rcro$ such that
$\dim R'_i(V) \leq 1$ or $\dim C'_i(V) \leq 1$.} \\
By Lemma \ref{Lemma1}, we lose no generality in assuming that
$C_n(\calV)=\F^n$ (as we can transpose our space and use row and column permutations).
Moreover, with the same line of reasoning as in the proof of Theorem \ref{DimensionInequalityTheorem}, we can find an index $i \in \lcro 1,n\rcro$
such that $R'_i(V)$ does not include $\Vect(E_{i,1},\dots,E_{i,n-1})$: indeed, otherwise
$\calV$ would contain a matrix of the form $\begin{bmatrix}
I_{n-1} & [0]_{(n-1) \times 1} \\
[?]_{1 \times (n-1)} & 1
\end{bmatrix}$, which is mapped to the non-zero scalar $f(\id_{\lcro 1,n\rcro})$ by $\widetilde{f}$.

Permuting rows, we see that no generality is lost in assuming that $R'_n(V)$
does not contain all the matrices $E_{n,1},\dots,E_{n,n-1}$.
Once more, we denote by $\calV'$ the affine subspace of $\calV$ consisting of its matrices with last column
$\begin{bmatrix}
[0]_{(n-1) \times 1} \\
1
\end{bmatrix}$ and we split every such matrix $M$ as
$$M=\begin{bmatrix}
P(M) & [0]_{(n-1) \times 1} \\
[?]_{1 \times (n-1)} & 1
\end{bmatrix} \quad \text{with $P(M) \in \Mat_{n-1}(\F)$.}$$
As $\Vect(E_{n,1},\dots,E_{n,n-1})\not\subset R'_n(V)$, we find that
$$\codim P(\calV') \leq \codim \calV-1=n-1,$$
and by Theorem \ref{DimensionInequalityTheorem} this yields $\codim P(\calV')=n-1$.
Denote by $W$ the translation vector space of $P(\calV')$. Then, by Lemma \ref{upperblocklemma}
the induction hypothesis applies to $P(\calV')$, which yields an index $i \in \lcro 1,n-1\rcro$ such that
$R'_i(W)=\{0\}$ or $C'_i(W)=\{0\}$.
From there, we see that $\dim R'_i(V) \leq 1$ or $\dim C'_i(V) \leq 1$:
indeed, let us assume that $R'_i(W)=\{0\}$, and let $\gamma$ map every $M \in R'_i(V)$ to its last entry; the kernel of $\gamma$ is included in the translation vector space of $\calV'$,
and for every matrix $M$ in this kernel we see from $R'_i(W)=\{0\}$ that the $i$-th row of $M$ must equal zero, whence $M=0$.
Hence, $\gamma$ is injective, which yields $\dim R'_i(V)\leq 1$. Likewise, $C'_i(W)=\{0\}$ implies $\dim C'_i(V) \leq 1$.

This completes our first step.

\vskip 2mm
\noindent \textbf{Step 2: There exists an index $i \in \lcro 1,n\rcro$ such that
$R'_i(V)=\{0\}$ or $C'_i(V)=\{0\}$.} \\
Assume that the contrary holds.
Then, by Step 1 and our starting assumptions, we find an index $i \in \lcro 1,n\rcro$ such that
$\dim R'_i(V)=1$ or $\dim C'_i(V)=1$.
Transposing if necessary, and using permutations of rows and columns,
we can reduce the situation to the one in which
$R'_n(V)$ is spanned by $N=\underset{k=p}{\overset{n}{\sum}} a_k E_{n,k}$ for some $p \in \lcro 1,n\rcro$
 and some list $(a_p,\dots,a_n)$ of non-zero scalars (note that here we entirely forget the intermediate reduced situation
 that was obtained in the proof of Step 1).
For $\tau \in  \mathfrak{S}_{n-1}$, we extend $\tau$ to an element $\overline{\tau}$ of $\mathfrak{S}_n$ such that $\overline{\tau}(n)=n$,
and we set $g : \tau \in \mathfrak{S}_{n-1} \mapsto f(\overline{\tau})$.

\vskip 2mm
\noindent \textbf{Substep 2.1: One has $p<n$.} \\
Assume on the contrary that $p=n$. \\
For any $M \in \calV$, we write
$$M=\begin{bmatrix}
P(M) & [?]_{(n-1) \times 1}\\
[?]_{1 \times (n-1)} & ?
\end{bmatrix} \quad \text{with $P(M) \in \Mat_{n-1}(\F)$.}$$
Then, one checks that
$$\forall M \in \calV, \; a_n\, \widetilde{g}(P(M))=\widetilde{f}(M+N)-\widetilde{f}(M)=0$$
Therefore, the affine space $P(\calV)$ is included in the null cone of $g$.
By Theorem \ref{DimensionInequalityTheorem}, this yields
$$\codim P(\calV) \geq n-1.$$
Yet, by the rank theorem
$$\codim \calV\geq \codim P(\calV)+(n-\dim R'_n(V)) \geq 2(n-1)>n,$$
contradicting our assumptions (note how we use the assumption that $n \geq 3$).

\vskip 2mm
\noindent \textbf{Substep 2.2: The final contradiction.} \\
Let us write every matrix $M \in \Mat_n(\F)$ as
$$M=\begin{bmatrix}
K(M) \\
[?]_{1 \times n}
\end{bmatrix} \quad \text{with $K(M) \in \Mat_{n-1,n}(\F)$.}$$
Then, by the rank theorem
$$\codim K(\calV)+\bigl(n-\dim R'_n(V)\bigr)=\codim \calV$$
and hence $\codim K(\calV)=1$.
Denote by $K(V)^\bot$ the right orthogonal complement of $K(\calV)$ for the bilinear form $(M,N)\in \Mat_{n-1,n}(\F) \times \Mat_{n,n-1}(\F)  \mapsto \tr(MN)$. Given $i \in \lcro 1,n\rcro$, if $C_i(K(\calV)) \neq \F^{n-1}$ then $K(V)^\bot$ contains
a nonzero matrix whose rows are all zero with the exception of the $i$-th. Since $\dim K(V)^\bot=1$ it follows that there is at most one such index $i$.
Since $p<n$, we recover an index $j \in \lcro p,n\rcro$
such that $K(\calV)$ contains a matrix whose $j$-th column equals zero.

Permuting columns, we see that no generality is lost in assuming that $j=n$. By combining this with the fact that
$V$ contains $N$, we deduce that $C_n(\calV)$ contains $\begin{bmatrix}
[0]_{(n-1) \times 1} \\
1
\end{bmatrix}$. Then, once more we consider the affine subspace $\calV'$ of all matrices of $\calV$ with last column
$\begin{bmatrix}
[0]_{(n-1) \times 1} \\
1
\end{bmatrix}$, and we write every matrix $M \in \calV'$ as
$$M=\begin{bmatrix}
P(M) & [0]_{(n-1) \times 1}\\
[?]_{1 \times (n-1)} & 1
\end{bmatrix} \quad \text{with $P(M) \in \Mat_{n-1}(\F)$.}$$
Since $R'_n(V)=\Vect(N)$, we now have $\dim P(\calV')=\dim \calV'$.
By combining Lemma \ref{upperblocklemma} with Theorem \ref{DimensionInequalityTheorem}, we find $\codim P(\calV') \geq n-1$.
We conclude that
$$\codim \calV \geq (n-1)+\codim P(\calV')=2(n-1)>n,$$
contradicting our basic assumptions.

Hence, we actually have an index $i \in \lcro 1,n\rcro$ such that
$R'_i(V)=\{0\}$ or $C'_i(V)=\{0\}$, as claimed. This completes the inductive step. Therefore, our proof by induction is complete.
\end{proof}

\subsection{Completing the proof of Theorem \ref{DimensionEqualityTheorem}}

Let $f : \mathfrak{S}_n \rightarrow \F^*$ be a normalized mapping, and
$V$ be a linear subspace of $\Mat_n(\F)$ with codimension $n$, that is included in $\calC(f)$.
If $n=1$ then $V=\{0\}$ and hence $V=\calV_{e_1}$ where $e_1$ denotes the first vector of the standard basis of $\F$. In the rest of the proof, we assume that $n \geq 2$.

By Lemma \ref{Lemma2}, there is an index $i \in \lcro 1,n\rcro$ such that $R'_i(V)=\{0\}$ or $C'_i(V)=\{0\}$.
First of all, we reduce the situation to the one where $C'_n(V)=\{0\}$.

Assume that $R'_i(V)=\{0\}$. Then, $V^T$ is included in $\calC(f^T)$
and it satisfies $C'_i(V^T)=\{0\}$. Moreover, $f^T$ is normalized. If there exists a column-$f^T$-adapted vector $X$ such that $\calV^T=\calV_X$, then $X$ is row-$f$-adapted and $\calV=(\calV_X)^T=\calV^X$.

Hence, it suffices to consider the case when $C'_i(V)=\{0\}$. Next, we reduce the situation to the one where
$i=n$. Choose a permutation $\tau$ of $\lcro 1,n\rcro$ such that $\tau(n)=i$.
Set $g : \sigma \in \mathfrak{S}_n \mapsto f(\sigma \tau^{-1})$, whose
associated Schur functional is $\widetilde{g} : M \mapsto \widetilde{f}(M P_{\tau}^{-1})$. Note that $g$ is normalized.
Set $V':=V P_\tau$.
For all $M \in V$, we have $\widetilde{g}(M P_\tau)=\widetilde{f}(M)=0$, whence $V' \subset \calC(g)$.
Given $M \in V'$ whose first $n-1$ columns equal zero, the columns of $M P_\tau^{-1}$ are zero with the possible exception of the $i$-th,
whence $M P_\tau^{-1}=0$ and finally $M=0$. Therefore, $C'_n(V')=\{0\}$.
Assume now that there is a column-$g$-adapted vector $Y \in \F^n$ such that $V'=\calV_Y$.
Then, $V=V' P_\tau^{-1}=\calV_X$ where $X:=P_\tau Y$. Let us write $X=(x_k)$ and $Y=(y_l)$.
For all indices $k,l$ such that $x_kx_l \neq 0$, we have $y_{\tau^{-1}(k)} y_{\tau^{-1}(l)} \neq 0$,
and hence $\tau^{-1}(k)$ is column-$g$-equivalent to $\tau^{-1}(l)$, and finally $k$ is column-$f$-equivalent to $l$.

Hence, in the remainder of the proof, it only remains to consider the case when $C'_n(V)=\{0\}$.
Since $\codim V=n$, this yields a linear mapping $h : \Mat_{n,n-1}(\F) \rightarrow \F^n$ such that
$$V=\Bigl\{\begin{bmatrix}
N & h(N)
\end{bmatrix} \mid N \in \Mat_{n,n-1}(\F)\Bigr\}.$$
Next, we analyse $h$. Let us write
$$h : N \in \Mat_{n,n-1}(\F) \mapsto \begin{bmatrix}
h_1(N) & \cdots & h_n(N)
\end{bmatrix}^T.$$
We shall prove that $h_i(N)$ is a function of the $i$-th row of $N$.

For $M \in \Mat_{n-1}(\F)$, denote by $\gamma(M)$ the image of
$\begin{bmatrix}
M \\
[0]_{1 \times (n-1)}
\end{bmatrix}$
under $h_n$. For $\tau \in \mathfrak{S}_{n-1}$, denote by $\overline{\tau}$ its extension as a permutation of $\lcro 1,n\rcro$,
and set $\overline{f}(\tau):=f(\overline{\tau})$.
Since $V$ is included in the null cone of $f$, we find that
$$\forall M \in \Mat_{n-1}(\F), \quad \gamma(M)\,\widetilde{\overline{f}}(M)=0.$$
Assume that $\gamma \neq 0$, and choose a nonzero element $a$ in the range of $\gamma$.
Set $\calW:=\gamma^{-1}\{a\}$, which is an affine hyperplane of $\Mat_{n-1}(\F)$. By the above,
$\calW$ is included in the null cone of $\overline{f}$.
Hence, by Lemma \ref{Lemma1}, $\dim \calW \leq (n-1)^2-(n-1)$, and it follows that $(n-1) \leq 1$, that is $n \leq 2$.
If $n=2$ then $\calW \subset \{0\}$, which is absurd.
We conclude that $\gamma=0$. In other words, $h_n$ vanishes at every matrix of
$\Mat_{n,n-1}(\F)$ whose $n$-th row equals zero. Likewise, one proves that, for all $i \in \lcro 1,n-1\rcro$,
the mapping $h_i$ vanishes at every matrix of $\Mat_{n,n-1}(\F)$ whose $i$-th row equals zero.
Hence, we have a matrix $B=(b_{i,j}) \in \Mat_{n,n-1}(\F)$ such that
$$\forall N \in \Mat_{n,n-1}(\F), \; \forall i \in \lcro 1,n\rcro, \; h_i(N)=\sum_{j=1}^{n-1} b_{i,j}\, n_{i,j.}$$
Hence, for all $M=(m_{i,j})$ in $V$, we find
$$0=\widetilde{f}(M)=\sum_{\sigma \in \mathfrak{S}_n} \Biggl(f(\sigma) \prod_{j=1}^{n-1} m_{\sigma(j),j}
\biggl(\sum_{k=1}^{n-1} b_{\sigma(n),k} \,m_{\sigma(n),k}\biggr)\Biggr),$$
and hence, for all $M=(m_{i,j})$ in $\Mat_{n,n-1}(\F)$,
$$0=\sum_{k=1}^{n-1} \sum_{\sigma \in \mathfrak{S}_n} \biggl(f(\sigma)\, b_{\sigma(n),k}\,m_{\sigma(n),k}\prod_{j=1}^{n-1} m_{\sigma(j),j}\biggr).$$
On the right hand-side of this equality, we see a polynomial function in the variables $m_{i,j}$, in which all the monomials have degree
at most $1$ in each of those variables. Hence, the corresponding formal polynomial equals zero, which yields the following result:
for all $k \in \lcro 1,n-1\rcro$ and all $\sigma \in \mathfrak{S}_n$,
\begin{equation}\label{subspaceidentity}
f(\sigma)\, b_{\sigma(n),k}+f(\sigma \tau_{k,n})\, b_{\sigma(k),k}=0.
\end{equation}
Now, fix $k \in \lcro 1,n-1\rcro$.
Assume first that $b_{i,k}=0$ for some $i \in \lcro 1,n\rcro$. For all $j \in \lcro 1,n\rcro \setminus \{i\}$, we can choose
$\sigma \in \mathfrak{S}_n$ such that $\sigma(k)=i$ and $\sigma(n)=j$, and hence \eqref{subspaceidentity} yields $b_{j,k}=0$
since $f$ vanishes nowhere. Therefore, either $b_{i,k}\neq 0$ for all $i \in \lcro 1,n\rcro$ or $b_{i,k}= 0$ for all $i \in \lcro 1,n\rcro$.
Assume now that the first case holds. Then,
$$\forall \sigma \in \mathfrak{S}_n, \; f(\sigma \tau_{k,n})=-\frac{b_{\sigma(n),k}}{b_{\sigma(k),k}}\,f(\sigma)$$
and hence $k$ and $n$ are column-$f$-equivalent. Since $f$ is normalized, it follows that $b_{\sigma(n),k}=b_{\sigma(k),k}$
for all $\sigma \in \mathfrak{S}_n$. Varying $\sigma$ yields $b_{i,k}=b_{1,k}$ for all $i \in \lcro 1,n\rcro$.
Hence, we have shown that, in any case $b_{i,k}=b_{1,k}$ for all $i \in \lcro 1,n\rcro$, and if $b_{1,k}$ is non-zero then
$k$ is column-$f$-equivalent to $n$.

To conclude, we define $X=(x_j)_{1 \leq j \leq n} \in \F^n$ by $x_j:=b_{1,j}$ for all $j \in \lcro 1,n-1\rcro$, and $x_n:=-1$.
It follows from the above that $V=\calV_X$ and that $X$ is column-$f$-adapted, which concludes the proof of Theorem \ref{DimensionEqualityTheorem}.


\section{A partial description of the $(f,g)$-transformations}\label{transfosection}

\subsection{Basic examples, main results}

Throughout the section, we let $f$ and $g$ be mappings from $\mathfrak{S}_n$ to $\mathbb{F}^*$.
For an endomorphism $U$ of the vector space $\Mat_n(\F)$, we define the condition
\begin{center}
$(\calC)$ : $\forall M \in \Mat_n(\F), \; \widetilde{g}(U(M))=0 \Leftrightarrow f(M)=0$.
\end{center}
In other words, $(\calC)$ is satisfied if and only if $U$ maps $\calC(f)$ into $\calC(g)$ and $\Mat_n(\F) \setminus \calC(f)$
into $\Mat_n(\F) \setminus \calC(g)$.
Note that this condition is satisfied whenever $U$ is an $(\alpha f,g)$-transformation for some $\alpha \in \F^*$.

Next, we give a basic example of such a map:

\begin{prop}
Let $f : \mathfrak{S}_n \rightarrow \F^*$ be a fully-normalized mapping, and
write $r(f)=(n_1,\dots,n_a)$ and $c(f)=(p_1,\dots,p_b)$.
Let $(P_1,\dots,P_a) \in \GL_{n_1}(\F)\times \cdots \times \GL_{n_a}(\F)$ and
$(Q_1,\dots,Q_b) \in \GL_{p_1}(\F)\times \cdots \times \GL_{p_b}(\F)$, and set
$P:=P_1 \oplus \cdots \oplus P_a \in \GL_n(\F)$, $Q:=Q_1 \oplus \cdots \oplus Q_b \in \GL_n(\F)$
and $\alpha:=\det P\det Q$.
Then, $M \in \Mat_n(\F) \mapsto PMQ$ is an $(\alpha f, f)$-transformation.
We say that it is a \textbf{standard $f$-similarity.}
\end{prop}

\begin{proof}
Remember that, for every positive integer $k>0$, the group $\GL_k(\F)$ is generated by the set consisting of the
dilation matrices, i.e.\ the diagonal matrices with exactly one non-zero entry, and the transvection matrices, i.e.\ the triangular matrices with diagonal entries all equal to $1$ and exactly one non-zero off-diagonal entry.
Using this, the situation is easily reduced to the one where all but one of $P_1,\dots,P_a,Q_1,\dots,Q_b$
are identity matrices, and the remaining one is a dilation matrix or a transvection matrix. Assume that this is the case,
and consider the situation where $P_1,\dots,P_a$ are identity matrices and there is a sole index $i \in \lcro 1,b\rcro$ for which
$Q_i$ is not an identity matrix.
\begin{itemize}
\item Assume first that $Q_i$ is a transvection matrix. Then, $Q=I_n+ \lambda E_{k,l}$ for some distinct column-$f$-equivalent indices
$k,l$ and some $\lambda \in \F$, whereas $P=I_n$. Let $M \in \Mat_n(\F)$. The matrix $PMQ=MQ$ is deduced from
$M$ by replacing the $l$-th column with $C_l(M)+\lambda C_k(M)$. Since $\widetilde{f}$ is $n$-linear with respect to the columns,
we deduce that $\widetilde{f}(PMQ)=\widetilde{f}(M)+\lambda \widetilde{f}(N)$, where $N$ is deduced from $M$
by replacing the $l$-th column by $C_k(M)$. Since $f$ is normalized, we find that $\widetilde{f}(N)=0$ and hence
$\widetilde{f}(PMQ)=\widetilde{f}(M)$. On the other hand, $\det P\det Q=1$.
\item Assume next that $Q_i$ is a dilation matrix. Then, $P=I_n$ and $Q$ is a dilation matrix whose factor we denote by $\alpha$.
Since $\widetilde{f}$ is linear with respect to each column, we readily find
$\forall M \in \Mat_n(\F), \; \widetilde{f}(PMQ)=\alpha \widetilde{f}(M)$, whereas $\det P\det Q=\alpha$.
\end{itemize}
With exactly the same line of reasoning, one deals with the case when all the $Q_j$'s are identity matrices and exactly one of the $P_i$'s is not an identity matrix, but a transvection matrix or a dilation matrix. This completes the proof.
\end{proof}

\begin{Def}
Let $f : \mathfrak{S}_n \rightarrow \F^*$ be a normalized mapping.
A permutation $\sigma$ of $\lcro 1,n\rcro$ is called \textbf{column-$f$-adapted} (respectively, \textbf{row-$f$-adapted})
whenever it maps any two column-$f$-equivalent indices (respectively, row-$f$-equivalent indices) to two column-$f$-equivalent indices (respectively, to two row-$f$-equivalent indices) and it is increasing on every column-$f$-equivalence class (respectively, on every row-$f$-equivalence class).
\end{Def}

The datum of a column-$f$-adapted permutation is equivalent to the one, for each integer $k>0$,
of a permutation of the set of all column-$f$-equivalence classes with cardinality $k$.
For example, if no distinct indices are column-$f$-equivalent, then every permutation of $\lcro 1,n\rcro$ is column-$f$-adapted, whereas
if all indices are column-$f$-equivalent then the sole column-$f$-adapted permutation is the identity.

\begin{Def}
Let $R \in \Mat_n(\F)$. We say that $R$ is \textbf{$f$-adapted} when $r_{i,j}=r_{i',j'}$ for all row-$f$-adapted indices $i$ and $i'$ and all column-$f$-adapted indices $j$ and $j'$.
We say that $R$ is \textbf{super-$f$-adapted} when it is $f$-adapted and all its entries in its first row and column equal $1$.
\end{Def}

In particular, if all indices in $\lcro 1,n\rcro$ are column-$f$-equivalent and all indices
in $\lcro 1,n\rcro$ are row-$f$-equivalent, then a matrix is $f$-adapted if and only if all its entries are equal.

Now, we are poised to state a partial result on the mappings that satisfy condition $(\calC)$.

\begin{theo}\label{fgtransfotheo}
Assume that $n \geq 2$.
Let $f : \mathfrak{S}_n \rightarrow \F^*$ and
$g : \mathfrak{S}_n \rightarrow \F^*$ be fully-normalized mappings, and let
$U : \Mat_n(\F) \rightarrow \Mat_n(\F)$ be a linear mapping such that
$$\forall M \in \Mat_n(\F), \; \widetilde{g}(U(M))=0 \Leftrightarrow \widetilde{f}(M)=0.$$
Then, exactly one of the following holds:
\begin{enumerate}[(a)]
\item There exists a unique quadruple $(K,\sigma,\tau,V)$ in which
$K \in \Mat_n(\F^*)$ is a super-$g$-adapted matrix, $\sigma$ is a row-$f$-adapted permutation of $\lcro 1,n\rcro$,
$\tau$ is a column-$f$-adapted permutation of $\lcro 1,n\rcro$, and $V$ is a standard $f$-similarity,
such that
$$\forall M \in \Mat_n(\F), \; U(M)=K \star (P_\sigma V(M) P_\tau).$$
\item There exists a unique quadruple $(K,\sigma,\tau,V)$ in which
$K \in \Mat_n(\F^*)$ is a super-$g$-adapted matrix, $\sigma$ is a column-$f$-adapted permutation of $\lcro 1,n\rcro$,
$\tau$ is a row-$f$-adapted permutation of $\lcro 1,n\rcro$, and $V$ is a standard $f^T$-similarity, such that
 $$\forall M \in \Mat_n(\F), \; U(M)=K \star (P_\sigma V(M^T) P_\tau).$$
\end{enumerate}
\end{theo}

\begin{cor}\label{fgtransfocor}
Assume that $n \geq 2$.
Let $f : \mathfrak{S}_n \rightarrow \F^*$ and
$g : \mathfrak{S}_n \rightarrow \F^*$, and let
$U : \Mat_n(\F) \mapsto \Mat_n(\F)$ be a linear mapping such that
$$\forall M \in \Mat_n(\F), \; \widetilde{g}(U(M))=0 \Leftrightarrow \widetilde{f}(M)=0.$$
Then:
\begin{enumerate}[(a)]
\item For some $\alpha \in \F^*$, the mapping $U$ is an $(\alpha f, g)$-transformation.
\item The mapping $g$ is PH-equivalent to $f$ or to $f^T$.
\end{enumerate}
\end{cor}

\begin{cor}\label{fgtransfocor2}
Assume that $n \geq 2$, and let $f : \mathfrak{S}_n \rightarrow \F^*$ be a rigid map.
Let $U$ be an $(f,f)$-transformation. Then, there exist a matrix $K \in \Mat_n(\F^*)$
and permutations $\sigma,\tau$ of $\lcro 1,n\rcro$ such that
$$U : M \mapsto K \star (P_\sigma MP_\tau) \quad \text{or} \quad
U : M \mapsto K \star (P_\sigma M^T P_\tau).$$
\end{cor}

A solution to the initial problem of determining all $(f,g)$-transformations would require that we determine all $(f,f)$-transformations.
The previous theorem gives an incomplete answer to the latter problem, and a full solution would require that if $f$ is fully-normalized we determine, for all
$\alpha \in \F^*$, for which triples $(K,\sigma,\tau)$, with an $f$-adapted matrix $K \in \Mat_n(\F^*)$, a
row-$f$-adapted permutation $\sigma \in \mathfrak{S}_n$, a column-$f$-adapted permutation $\tau \in \mathfrak{S}_n$, the linear bijection
$M \mapsto K \star (P_\sigma\, M\, P_\tau)$ (or $M \mapsto K \star (P_\tau\, M^T\, P_\sigma)$) is an $(f,f)$-transformation.
We doubt that a general neat description exists beyond this point. The case when $f$ is central is completely solved in Section \ref{centralsection}, however.

\vskip 3mm
The remainder of the present section is laid out as follows:
\begin{itemize}
\item In Section \ref{injsection}, it is proved that every linear map that satisfies condition ($\calC$) is an automorphism of $\Mat_n(\F)$.
\item We prove a portion of Theorem \ref{fgtransfotheo} in Section \ref{analysissection}:
we establish the ``existence" part by examining the effect of $U$ on the linear subspaces of codimension $n$ included in $\calC(f)$,
and the effect of $U^{-1}$ on the linear subspaces of codimension $n$ included in $\calC(g)$.
\item In Section \ref{uniquenesssection}, we complete the proof of Theorem \ref{fgtransfotheo} by tackling the uniqueness statements.
\item Corollaries \ref{fgtransfocor} and \ref{fgtransfocor2} are drawn in Sections \ref{cor1section} and \ref{cor2section}, respectively.
\item Finally, in Section \ref{detperpressection}, we show how Theorem \ref{fgtransfotheo} yields Theorems \ref{detpreservers} and
\ref{perpreservers} with limited additional effort.
\end{itemize}

\subsection{Injectivity}\label{injsection}

\begin{lemma}\label{injectlemma}
Assume that $U$ satisfies condition $(\calC)$. Then, $U$ is injective.
\end{lemma}

\begin{proof}
Assume on the contrary that the kernel of $U$ contains a non-zero matrix $A=(a_{i,j})$.
Then, for all $M \in \calC(f)$, we successively find
$\widetilde{g}(U(M))=0$, $\widetilde{g}(U(A+M))=0$ and $\widetilde{f}(A+M)=0$.
Hence, $\calC(f)$ is stable under the translation $M \mapsto A+M$.

Let us find a contradiction from here. Replacing $f$ with a PH-equivalent mapping, we can assume that $a_{n,n} \neq 0$.
Then, we define $B \in \Mat_n(\F)$ as the matrix in which the first $n-1$ columns equal those of $I_n-A$, and the last one is zero.
Hence $\widetilde{f}(B)=0$ whereas $A+B=\begin{bmatrix}
I_{n-1} & [?]_{(n-1) \times 1} \\
[0]_{1 \times (n-1)} & a_{n,n}
\end{bmatrix}$. Hence, $\widetilde{f}(A+B)=f(\id)\, a_{n,n} \neq 0$.
This contradicts an earlier result, and we conclude that $U$ is injective.
\end{proof}

This yields Theorem \ref{automtheo}.

\subsection{An explicit shape for $U$}\label{analysissection}

Let $n \geq 2$ be an integer.
Let $f : \mathfrak{S}_n \rightarrow \F^*$ and
$g : \mathfrak{S}_n \rightarrow \F^*$ be fully-normalized mappings, and let
$U : \Mat_n(\F) \rightarrow \Mat_n(\F)$ be a linear mapping such that
$$\forall M \in \Mat_n(\F), \; \widetilde{g}(U(M))=0 \Leftrightarrow \widetilde{f}(M)=0.$$
We wish to prove that $U$ has one of the two shapes described in Theorem \ref{fgtransfotheo}.

\begin{claim}\label{directim}
Let $X$ be a column-$f$-adapted vector of $\F^n$.
Then, either there exists a column-$g$-adapted vector $Y$ of $\F^n$ such that
$U(\calV_X)=\calV_Y$, or there exists a row-$g$-adapted vector $Y$ of $\F^n$ such that
$U(\calV_X)=\calV^Y$.

Let $X'$ be a row-$f$-adapted vector of $\F^n$.
Then, either there exists a column-$g$-adapted vector $Z$ of $\F^n$ such that
$U(\calV^{X'})=\calV_Z$, or there exists a row-$g$-adapted vector $Z$ of $\F^n$ such that
$U(\calV^{X'})=\calV^Z$.
\end{claim}

\begin{proof}
We know that $\calV_X$ is a linear subspace of $\calC(f)$ with codimension $n$ in $\Mat_n(\F)$.
Since $U$ is an automorphism of $\Mat_n(\F)$ (see Lemma \ref{injectlemma}), we deduce from condition ($\calC$) that
$U(\calV_X)$ is a linear subspace of $\calC(g)$ with codimension $n$ in $\Mat_n(\F)$.
The first statement then follows from Theorem \ref{DimensionEqualityTheorem}.

The second statement is proved in the same manner.
\end{proof}

Applying this to $U^{-1}$, which satisfies condition ($\calC$) for the pair $(g,f)$, we obtain:

\begin{claim}\label{directimbis}
Let $X$ be a column-$g$-adapted vector of $\F^n$.
Then, either there exists a column-$f$-adapted vector $Y$ of $\F^n$ such that
$U^{-1}(\calV_X)=\calV_Y$, or there exists a row-$f$-adapted vector $Y$ of $\F^n$ such that
$U^{-1}(\calV_X)=\calV^Y$.

Let $X'$ be a row-$g$-adapted vector of $\F^n$.
Then, either there exists a column-$f$-adapted vector $Z$ of $\F^n$ such that
$U^{-1}(\calV^{X'})=\calV_Z$, or there exists a row-$f$-adapted vector $Z$ of $\F^n$ such that
$U^{-1}(\calV^{X'})=\calV^Z$.
\end{claim}

Denote by $e_1$ the first vector of the standard basis of $\F^n$ (it is column-$g$-adapted).
Assume that $U^{-1}(\calV_{e_1})=\calV^{Y_0}$ for some row-$f$-adapted vector $Y_0$.
Set $U' : M \in \Mat_n(\F) \mapsto U(M^T)$ and note that $U'$ satisfies condition ($\calC$) for the pair $(f^T,g)$.
Note also that $(U')^{-1}(\calV_{e_1})=\calV_{Y_0}$ and that $Y_0$ is column-$f^T$-adapted.
If we prove that $U'$ has one of the two possible shapes claimed in Theorem \ref{fgtransfotheo},
then it is obvious that so does $U$.

Therefore, no generality is lost in making the following additional assumption:
\begin{center}
$U^{-1}(\calV_{e_1})=\calV_{Y_0}$ for some column-$f$-adapted vector $Y_0$.
\end{center}

Noting that the problem is unchanged by right-composing $U$
with standard $f$-similarities (the set of all standard $f$-similarities is obviously a subgroup of the group of all automorphisms of the vector space
$\Mat_n(\F)$), we will, after several such compositions, slowly reduce the situation to the one where the properties of $U$
are ever simpler, until we find a mapping of the form $M \mapsto K \star (P_\sigma MP_\tau)$ for some super-$f$-adapted matrix $K$, some
row-$f$-adapted permutation $\sigma$ and some column-$f$-adapted permutation $\tau$.

In order to move forward, we need two additional lemmas, whose proofs are easy:

\begin{lemma}
Let $X$ and $Y$ be non-zero vectors of $\F^n$. Then $\calV_X \cap \calV^Y$ has codimension $2n-1$ in $\Mat_n(\F)$.
\end{lemma}

\begin{lemma}\label{ranklemma}
Let $X_1,\dots,X_p$ be non-zero vectors of $\F^n$. Denote by $r$ the rank of $(X_1,\dots,X_p)$.
Then, both $\underset{i=1}{\overset{p}{\bigcap}} \calV_{X_i}$ and
$\underset{i=1}{\overset{p}{\bigcap}} \calV^{X_i}$ have codimension $nr$ in $\Mat_n(\F)$.
\end{lemma}

\begin{claim}\label{compatrowclaim}
For every row-$f$-adapted vector $X$, there is a row-$g$-adapted vector $Y$
such that $U(\calV^X)=\calV^Y$.
\end{claim}

\begin{proof}
Let $X \in \F^n$ be row-$f$-adapted. Assume that $U(\calV^X)=\calV_Y$ for some column-$g$-adapted vector
$Y$. Then, as $U$ is injective, the space $U(\calV^X \cap \calV_{Y_0})=\calV_Y \cap \calV_{e_1}$ has codimension either $n$ or $2n$ in $\Mat_n(\F)$,
contradicting the fact that $U$ is an automorphism of $\Mat_n(\F)$ and $\calV^X \cap \calV_{Y_0}$ has codimension
$2n-1$ in $\Mat_n(\F)$. Hence, by Claim \ref{directim}, we find that $U(\calV^X)=\calV^Y$ for some row-$g$-adapted vector
$Y$.
\end{proof}

With the same line of reasoning, Claim \ref{compatrowclaim} yields:

\begin{claim}
For every column-$f$-adapted vector $X$, there is a column-$g$-adapted vector $Y$
such that $U(\calV_X)=\calV_Y$.
\end{claim}

Noting that $U(\calV_{Y_0})=\calV_{e_1}$, we apply the above line of reasoning to $U^{-1}$ to obtain:

\begin{claim}
For every column-$g$-adapted vector $X$, there is a column-$f$-adapted vector $Y$
such that $U^{-1}(\calV_X)=\calV_Y$.

For every row-$g$-adapted vector $X$, there is a row-$f$-adapted vector $Y$
such that $U^{-1}(\calV^X)=\calV^Y$.
\end{claim}

\begin{Def}
Two column-$f$-adapted vectors $X$ and $X'$ of $\F^n$ are called \textbf{column-$f$-connected}
whenever the union of the supports of $X$ and $X'$ is included in a column-$f$-equivalence class.
\end{Def}

\begin{claim}
Let $X,X'$ be column-$f$-adapted vectors that are column-$f$-connected.
Then, there exist column-$g$-adapted vectors $Y$ and $Y'$ that are column-$g$-connected
and such that $U(\calV_X)=\calV_Y$ and $U(\calV_{X'})=\calV_{Y'}$.
\end{claim}

\begin{proof}
We already know that there are column-$g$-adapted vectors $Y$ and $Y'$
such that $U(\calV_X)=\calV_Y$ and $U(\calV_{X'})=\calV_{Y'}$.
It remains to prove that $Y$ and $Y'$ are column-$g$-connected.
This is obvious if $X$ and $X'$ are collinear: in that case indeed, $\calV_Y=\calV_{Y'}$ leads
to $Y$ being collinear with $Y'$ (use Lemma \ref{ranklemma}, for example).
Assume now that $X$ and $X'$ are not collinear, so that $\calV_X \neq \calV_{X'}$ and hence $\calV_Y\neq \calV_{Y'}$,
whence $Y$ and $Y'$ are not collinear.

Noting that $\calV_X \cap \calV_{X'} \subset \calV_{X+X'}$, we see that $\calV_Y \cap \calV_{Y'} \subset U(\calV_{X+X'})$
whence $\calV_Y \cap \calV_{Y'} \subset \calV_{Z}$ for some column-$g$-adapted vector $Z$
that is neither collinear with $Y$ nor with $Y'$.
Hence $\calV_Y \cap \calV_{Y'}=\calV_{Y} \cap \calV_{Y'} \cap \calV_{Z}$ and we deduce from Lemma \ref{ranklemma} that
$(Y,Y',Z)$ has rank $2$. Hence, $Z$ is a linear combination of $Y$ and $Y'$ with nonzero coefficients. If the respective supports of $Y$ and $Y'$
were not included in the same column-$g$-equivalence class, then the support of $Z$ would be their union, and obviously it would
not be included in a column-$g$-equivalence class. Hence, $Y$ is column-$g$-connected to $Y'$.
\end{proof}

In order to simplify the discourse in the rest of the proof, we will call a \textbf{line} a $1$-dimensional \emph{linear} subspace
of a vector space.

We see that $\calV_X$ depends only on the line $d:=\F X$, and we shall also write it as $\calV_d$.
For a vector space $W$, denote by $\Pgros(W)$ the corresponding projective space (i.e.\ the set of all lines in $W$). Denote by $\Pgros_f(\F^n)$ the set of all lines that are spanned by column-$f$-adapted vectors. In other words, if we denote by $(e_1,\dots,e_n)$ the standard basis of $\F^n$ and by
$\mathcal{O}$ the set of all column-$f$-equivalence classes,
$$\Pgros_f(\F^n)=\underset{A \in \mathcal{O}}{\bigcup} \Pgros(\Vect(e_i)_{i \in A}).$$

Part of the above results is then summed up as follows:

\begin{claim}
There is a uniquely-defined mapping $\varphi : \Pgros_f(\F^n) \rightarrow \Pgros_g(\F^n)$ such that
$U(\calV_d)=\calV_{\varphi(d)}$ for all $d \in \Pgros_f(\F^n)$.

Moreover, the map $\varphi$ is \textbf{$(f,g)$-coherent} in the following sense:
for every column-$f$-equivalence class $A$, there is a uniquely-defined column-$g$-equivalence class $B$ such that
$\varphi$ maps $\Pgros(\Vect(e_i)_{i \in A})$ into $\Pgros(\Vect(e_i)_{i \in B})$.
\end{claim}

Applying this to $U^{-1}$, we also obtain:

\begin{claim}
There is a uniquely-defined $(g,f)$-coherent mapping $\psi : \Pgros_g(\F^n) \rightarrow \Pgros_f(\F^n)$ such that
$U^{-1}(\calV_{d'})=\calV_{\psi(d')}$ for all $d' \in \Pgros_g(\F^n)$.
\end{claim}

We deduce:

\begin{claim}
The mappings $\varphi$ and $\psi$ are bijections, inverse to one another.
Moreover, denoting by $\calO_{c,f}$ the set of all column-$f$-equivalence classes, and by $\calO_{c,g}$ the set of all
column-$g$-equivalence classes, there is a unique bijection $\Sigma : \calO_{c,g} \rightarrow \calO_{c,f}$
such that, for all $O \in \calO_{c,f}$, the mapping $\psi$ maps the projective space $\Pgros(\Vect(e_i)_{i \in O})$
bijectively onto $\Pgros(\Vect(e_j)_{j \in \Sigma(O)})$.
\end{claim}

Next, we obtain more precise information on $\psi$:

\begin{claim}\label{independenceclaim}
The mapping $\psi$ preserves linear independence for finite families of lines.
\end{claim}

\begin{proof}
Let $d_1,\dots,d_k$ be independent lines of $\F^n$ spanned by column-$g$-adapted vectors. For every $i \in \lcro 1,k\rcro$, let us choose a non-zero vector $X_i \in d_i$
and a column-$f$-adapted vector $Y_i \in \F^n$ such that $U^{-1}(\calV_{X_i})=\calV_{Y_i}$.
Hence, $\underset{i=1}{\overset{k}{\bigcap}} \calV_{X_i}$ has codimension $nk$ in $\Mat_n(\F)$.
Since $U$ is an automorphism of $\Mat_n(\F)$, it follows that
$\underset{i=1}{\overset{k}{\bigcap}} \calV_{Y_i}=U^{-1}\Bigl(\underset{i=1}{\overset{k}{\bigcap}} \calV_{X_i}\Bigr)$ has codimension $nk$ in $\Mat_n(\F)$. It follows from Lemma \ref{ranklemma} that $(Y_1,\dots,Y_k)$ has rank $k$, i.e.\ $\psi(d_1),\dots,\psi(d_k)$ are independent.
\end{proof}

Of course, the same applies to $\varphi$.
By comparing the dimensions and by using the fact that $f$ and $g$ are fully-normalized, we deduce:

\begin{claim}
For all $O \in \calO_{c,g}$,
one has $|\Sigma(O)|=|O|$. Hence, $\calO_{c,g}=\calO_{c,f}$ and there is a (unique) column-$f$-adapted permutation $\tau$ of $\lcro 1,n\rcro$ such that $\Sigma^{-1}(O)$ is the direct image of $O$ under $\tau$ for all $O \in \calO_{c,f}$.
\end{claim}

Next, let $O \in \calO_{c,f}$. We write $O=\lcro a,b\rcro$ and $\Sigma^{-1}(O)=\lcro c,d\rcro$.
By the above, we have a basis $(y_a,\dots,y_b)$ of $\Vect(e_j)_{j \in \Sigma^{-1}(O)}$ such that
$U^{-1}(\calV_{e_i})=\calV_{y_i}$ for all $i \in \lcro a,b\rcro$.
We define $Q_O \in \GL_{|O|}(\F)$ as the matrix of the basis $(e_c,\dots,e_d)$ in the basis
$(y_a,\dots,y_b)$ of $\Vect(e_j)_{j \in \Sigma^{-1}(O).}$
Then, in writing $O_1,\dots,O_p$ the elements of $\calO_{c,g}$ in that order
(so that, whenever $i<j$, every element of $O_i$ is less than every element of $O_j$),
we set $Q:=Q_{\Sigma(O_1)} \oplus \cdots \oplus Q_{\Sigma(O_p)}$,
so that $M \mapsto MQ$ is a standard $f$-similarity.
Fixing once again $O \in \calO_{c,f}$ and writing $O=\lcro a,b\rcro$ and
$\Sigma^{-1}(O)=\lcro c,d\rcro$, we have a basis $(y_a,\dots,y_b)$ of $\Vect(e_j)_{c \leq j \leq d}$ such that
$U^{-1}(\calV_{e_i})=\calV_{y_i}$ for all $i\in \lcro a,b\rcro$,
and
$$\forall i \in \lcro a,b\rcro, \; Qy_i=e_{\tau(i)}.$$
Hence, given $i \in \lcro a,b\rcro$, we have, for all $M \in \Mat_n(\F)$,
$$U(MQ)e_i=0 \Leftrightarrow MQ \in U^{-1}(\calV_{e_i})
\Leftrightarrow MQy_i=0 \Leftrightarrow Me_{\tau(i)}=0.$$
By right-composing $U$ with the standard $f$-similarity $M \mapsto MQ$,
we are then reduced to the case when
$$\forall i \in \lcro 1,n\rcro, \; U^{-1}(\calV_{e_i})=\calV_{e_{\tau(i)}.}$$

Likewise, we find that $r(f)=r(g)$ and, after replacing $U$ with $M \mapsto U(PM)$ for a well-chosen matrix $P \in \GL_n(\F)$
such that $M \mapsto PM$ is a standard $f$-similarity, we reduce the situation further to the one where, in addition to
the above properties, we have a row-$f$-adapted permutation $\sigma$ of $\lcro 1,n\rcro$ such that
$$\forall i \in \lcro 1,n\rcro, \; U^{-1}(\calV^{e_i})=\calV^{e_{\sigma(i)}}.$$

We are now close to the conclusion, but a few additional rounds of reduction are nevertheless necessary!
Let $i \in \lcro 1,n\rcro$. For all $M \in \Mat_n(\F)$ with $C_{i}(M)=0$, we know that
$C_{\tau^{-1}(i)}(U(M))=0$. This yields a linear mapping $\varphi_i : \F^n \rightarrow \F^n$ such that
$$\forall M \in \Mat_n(\F), \; C_{\tau^{-1}(i)}(U(M))=\varphi_i(C_i(M)).$$
Since $U$ is surjective, we find that $\varphi_i$ is surjective, and hence it is an automorphism of $\F^n$.

\begin{claim}
Let $i,j$ be column-$f$-equivalent indices. Then, the mappings $\varphi_i$ and $\varphi_j$ are collinear.
\end{claim}

\begin{proof}
It suffices to consider the case when $i \neq j$. Set $x:=e_i-e_j$, which is column-$f$-adapted.
We know that $U(\calV_{e_i})=\calV_{e_{\tau^{-1}(i)}}$ and $U(\calV_{e_j})=\calV_{e_{\tau^{-1}(j)}}$.
Hence, by Claim \ref{independenceclaim}, we find that $U(\calV_{x})=\calV_y$ for some non-zero vector $y \in \Vect(e_{\tau^{-1}(i)},e_{\tau^{-1}(j)})$.
Let us write $y=\lambda e_{\tau^{-1}(i)}+\mu e_{\tau^{-1}(j)}$ for some non-zero pair $(\lambda,\mu)\in \F^2$.
Let $X \in \F^n$. We can choose a matrix $M \in \Mat_n(\F)$ whose $i$-th and $j$-th column both equal $X$.
Thus, $M \in \calV_x$ and hence $U(M) \in \calV_y$, which yields $\lambda \varphi_i(X)+\mu \varphi_j(X)=0$.
The claimed statement follows.
\end{proof}

Likewise, we obtain, for each $i \in \lcro 1,n\rcro$, a linear bijection $\psi_i : \F^n \rightarrow \F^n$ such that
$$\forall M \in \Mat_n(\F), \; R_{\sigma^{-1}(i)}(U(M))=\psi_i(R_i(M)).$$
Moreover, given row-$f$-equivalent indices $i$ and $j$, the mappings $\psi_i$ and $\psi_j$ are collinear.

Next, we will further reduce the situation to the one where $\varphi_i=\varphi_j$ (respectively, $\psi_i=\psi_j$)
whenever $i,j$ are column-$f$-equivalent (respectively, row-$f$-equivalent) indices.
To do so, for all $i \in \lcro 1,n\rcro$, we find the minimal index $p$ in the column-$f$-equivalence class of $i$,
and we denote by $\lambda_i$ the sole non-zero scalar such that $\varphi_i=\lambda_i \varphi_p$
(so that $\lambda_i=1$ if $i=p$). Likewise, for all $i \in \lcro 1,n\rcro$, we
find the minimal index $q$ in the row-$f$-equivalence class of $i$,
and we denote by $\mu_i$ the sole non-zero scalar such that $\psi_i=\mu_i \psi_q$
(so that $\mu_i=1$ if $i=q$).
Set $D:=\Diag(\mu_1^{-1},\dots,\mu_n^{-1})$ and $\Delta:=\Diag(\lambda_1^{-1},\dots,\lambda_n^{-1})$.
Then, by right-composing $U$ with the standard $f$-similarity $M \mapsto DM\Delta$, we preserve all the previous assumptions and results,
but now we have the additional properties:
\begin{itemize}
\item[(i)] For all column-$f$-equivalent indices $i,j$, one has $\varphi_i=\varphi_j$.
\item[(ii)] For all row-$f$-equivalent indices $i,j$, one has $\psi_i=\psi_j$.
\end{itemize}

The conclusion is near.
Let $(i,j)\in \lcro 1,n\rcro^2$. The unit matrix $E_{i,j}$ belongs to $\calV_{e_k}$ for all $k \in \lcro 1,n\rcro \setminus \{j\}$,
whence $U(E_{i,j})$ belongs to $\calV_{e_{\tau^{-1}(k)}}$ for all $k \in \lcro 1,n\rcro \setminus \{j\}$.
Likewise, $U(E_{i,j})$ belongs to $\calV^{e_{\sigma^{-1}(k)}}$ for all $k \in \lcro 1,n\rcro \setminus \{i\}$.
It follows that $U(E_{i,j})=\lambda_{i,j} E_{\sigma^{-1}(i),\tau^{-1}(j)}$ for some
scalar $\lambda_{i,j} \in \F$, which is non-zero since $U$ is injective.

The following result follows from properties (i) and (ii) in the above:

\begin{claim}\label{claimconstantlambda}
Let $i,i',j,j'$ belong to $\lcro 1,n\rcro$.
\begin{enumerate}[(i)]
\item If $i$ is row-$f$-equivalent to $i'$, then $\lambda_{i,j}=\lambda_{i',j}$.
\item If $j$ is column-$f$-equivalent to $j'$, then $\lambda_{i,j}=\lambda_{i,j'}$.
\end{enumerate}
\end{claim}

Now, we reduce the situation further to the one where $\lambda_{\sigma(1),\tau(i)}=\lambda_{\sigma(i),\tau(1)}=1$
for all $i \in \lcro 1,n\rcro$. First of all, by setting $\beta:=\lambda_{\sigma(1),\tau(1)}^{-1}$, by noting that $M \mapsto \beta M$ is a standard $f$-similarity, and by replacing $U$ with $U \mapsto U(\beta M)$, we reduce the situation to the one where
$\lambda_{\sigma(1),\tau(1)}=1$. Next, denote by
$P' \in \GL_n(\F)$ the diagonal matrix with diagonal entries $\lambda_{1,\tau(1)}^{-1},\dots,
\lambda_{n,\tau(1)}^{-1}$, and by
$Q' \in \GL_n(\F)$ the diagonal matrix with diagonal entries $\lambda_{\sigma(1),1}^{-1},\dots,
\lambda_{\sigma(1),n}^{-1}$.

Replacing $U$ with $M \mapsto U(P'MQ')$, we see that all the previous assumptions and properties are untouched, but in this new situation $\lambda_{\sigma(1),i}=\lambda_{i,\tau(1)}=1$
for all $i \in \lcro 1,n\rcro$.

Since $\sigma$ is row-$f$-adapted and $\tau$ is column-$f$-adapted, it follows from this last property and from Claim \ref{claimconstantlambda} that the matrix
$$K:=(\lambda_{\sigma(i),\tau(j)})_{1\leq i,j \leq n} \in \Mat_n(\F^*)$$
is super-$f$-adapted, and we have shown that
$$\forall M \in \Mat_n(\F), \; U(M)=K \star (P_{\sigma^{-1}} M P_{\tau}).$$
Obviously, $\sigma^{-1}$ is row-$f$-adapted, and we have the expected conclusion at last.

\subsection{Uniqueness}\label{uniquenesssection}

Here, we conclude the proof of Theorem \ref{fgtransfotheo} by tackling the uniqueness statements.
Let $f$, $g$, $U$ satisfy the assumptions of Theorem \ref{fgtransfotheo}. We have to prove that
$U$ cannot be of both the forms mentioned in that theorem, and we need to prove the uniqueness of the quadruple
$(K,\sigma,\tau,V)$ in each case.

First of all, let $P,Q$ be matrices of $\GL_n(\F)$, let $\sigma$ and $\tau$ be elements of $\mathfrak{S}_n$,
and let $K$ be a matrix of $\Mat_n(\F^*)$. Assume that
$$U : M \mapsto K \star (P_\sigma(PMQ) P_\tau) \quad \text{or} \quad U : M \mapsto K \star (P_\sigma(PM^TQ) P_\tau).$$
In the first case, for all $M \in \Mat_n(\F)$,
$$U(M) \in \calV_{e_1} \Leftrightarrow P_\sigma(PMQ) P_\tau \in \calV_{e_1}
\Leftrightarrow MQe_{\tau(1)}=0 \Leftrightarrow
M \in \calV_{Qe_{\tau(1)}},$$
and hence $U^{-1}(\calV_{e_1})=\calV_{Q e_{\tau(1)}}$. In the second case we obtain likewise
$U^{-1}(\calV_{e_1})=\calV^{Q e_{\tau(1)}}$. However, since $n \geq 2$ there do not exist non-zero vectors $z,z'$
of $\F^n$ such that $\calV_z=\calV^{z'}$, and hence only one case is possible.

Next, in order to demonstrate the uniqueness of the quadruple $(K,\sigma,\tau,V)$ in each case,
it obviously suffices to prove the following result:

\begin{lemma}
Let $f : \mathfrak{S}_n \rightarrow \F^*$ be a fully-normalized mapping.
Let $K,K'$ be super-$f$-adapted matrices of $\Mat_n(\F^*)$, $\sigma,\sigma'$ be row-$f$-adapted permutations of $\lcro 1,n\rcro$,
$\tau,\tau'$ be column-$f$-adapted permutations of $\lcro 1,n\rcro$, and $V,V'$ be standard $f$-similarities such that
$$\forall M \in \Mat_n(\F), \; K \star (P_\sigma V(M) P_\tau)=K' \star (P_{\sigma'} V'(M) P_{\tau'}).$$
Then, $(K,\sigma,\tau,V)=(K',\sigma',\tau',V')$.
\end{lemma}

\begin{proof}
Obviously,
the mapping $W:=V \circ (V')^{-1}$ is a standard $f$-similarity, and the matrix
$K'':=K'\star K^{[-1]}$ is super-$f$-adapted. We have
$$\forall M \in \Mat_n(\F), \; P_\sigma W(M) P_\tau=K'' \star (P_{\sigma'} M P_{\tau'}).$$
Next,
$$\forall M \in \Mat_n(\F), \; K'' \star (P_{\sigma'} M P_{\tau'})=P_{\sigma'}(L \star M) P_{\tau'}$$
where $L:=P_{\sigma'}^{-1} K'' P_{\tau'}^{-1}$. Hence, by setting $\sigma'':=(\sigma')^{-1}  \sigma$
and $\tau'':=\tau  (\tau')^{-1}$, we obtain the identity
\begin{equation}\label{identuniqueness}
\forall M \in \Mat_n(\F), \; P_{\sigma''} W(M) P_{\tau''}=L \star M.
\end{equation}
In order to conclude, it suffices to demonstrate that $\sigma''=\id=\tau''$, that $W=\id_{\Mat_n(\F)}$, and that all the entries of $L$ equal $1$. Indeed, all the entries of $K''$ will then be equal to $1$, leading to $K=K'$,
all the while $V=V'$ and $(\sigma,\tau)=(\sigma',\tau')$.

Now, let us write $c(f)=(n_1,\dots,n_q)$ and $r(f)=(m_1,\dots,m_p)$, so that $W$ reads
$M \mapsto PMQ$, where $P$ has the form $P_1 \oplus \cdots \oplus P_p$ for a list
$(P_1,\dots,P_p) \in \GL_{m_1}(\F) \times \cdots \times \GL_{m_p}(\F)$, and
$Q$ has the form $Q_1 \oplus \cdots \oplus Q_q$ for a list
$(Q_1,\dots,Q_q) \in \GL_{n_1}(\F) \times \cdots \times \GL_{n_q}(\F)$.

Next, let $i \in \lcro 1,n\rcro$. Judging from identity \eqref{identuniqueness}, we find, for all $M \in \Mat_n(\F)$,
$$Me_i=0 \Leftrightarrow (L \star M)e_i=0 \Leftrightarrow My=0,$$
where $y:=Q e_{\tau''(i)}$. It follows that $y$ is collinear with $e_i$. However, the shape of $Q$ shows that
$Q e_{\tau''(i)}$ is a linear combination of vectors of the form $e_j$ where $j$ is column-$f$-equivalent to $\tau''(i)$.
Hence, $i$ is column-$f$-equivalent to $\tau''(i)$, which yields $\tau''(i)=i$ since $\tau''$ is column-$f$-adapted.
Hence $\tau''$ is the identity of $\lcro 1,n\rcro$. It follows that $Qe_i$ is collinear with $e_i$ for all $i \in \lcro 1,n\rcro$.
Likewise, we obtain that $\sigma''$ is the identity of $\lcro 1,n\rcro$ and that $P^T e_i$ is collinear with $e_i$ for all $i \in \lcro 1,n\rcro$. Hence, $P$ and $Q$ are diagonal matrices (with non-zero diagonal entries). Writing the corresponding diagonal vectors as $X$ and $Y$, we find $\forall M \in \Mat_n(\F), \; W(M)=(X Y^T) \star M$, and by applying this to the vectors of the standard basis of $\Mat_n(\F)$ we find $XY^T=L$. Yet, one of the columns of $L$ has all its entries equal, whence all the entries of
$X$ are equal; likewise since one of the rows of $L$ has all its entries equal, all the entries of $Y$ are equal.
Hence $W : M \mapsto \lambda M$ for some non-zero scalar $\lambda$, and all the entries of $L$ equal $\lambda$.
Finally $\lambda=1$ since one of the entries of $L$ equals $1$. We conclude that $W=\id_{\Mat_n(\F)}$ and $L=(1)_{1 \leq i,j \leq n}$,
which completes the proof.
\end{proof}

\subsection{Proof of Corollary \ref{fgtransfocor}}\label{cor1section}

Here, we prove Corollary \ref{fgtransfocor}. To this end, we need a preliminary result on
null cones:

\begin{prop}\label{coneprop}
Let $f$ and $g$ be mappings from $\mathfrak{S}_n$ to $\F^*$.
Then, $\calC(f)=\calC(g)$ if and only if $g=\alpha f$ for some non-zero scalar $\alpha$.
\end{prop}

\begin{proof}
The converse implication is obvious. Assume that $\calC(f)=\calC(g)$.
Multiplying $g$ with $\frac{f(\id)}{g(\id)}$, we lose no generality in assuming that $f(\id)=g(\id)$, in which case we aim at proving that $g=f$.
Let $\sigma \in \mathfrak{S}_n$, and let $i<j$ be elements of $\lcro 1,n\rcro$. Set $\tau:=\tau_{i,j}$.
Let $\lambda \in \F$, and consider the matrix $M=(m_{k,l}) \in \Mat_n(\F)$ defined as follows:
$$m_{k,l}:=\begin{cases}
1 & \text{if $k=\sigma(l)$} \\
-\lambda & \text{if $l=i$ and $k=\sigma(j)$} \\
1 & \text{if $l=j$ and $k=\sigma(i)$} \\
0 & \text{otherwise.}
\end{cases}$$
One checks that
$$\widetilde{f}(M)=f(\sigma)-\lambda f(\sigma \tau) \quad \text{and} \quad \widetilde{g}(M)=g(\sigma)-\lambda g(\sigma \tau).$$
Hence, for all $\lambda \in \F$, we have
$$f(\sigma)-\lambda f(\sigma \tau)=0 \Leftrightarrow g(\sigma)-\lambda g(\sigma \tau)=0.$$
Since $f(\sigma \tau) \neq 0$ and $g(\sigma \tau) \neq 0$, this yields
$$\frac{g(\sigma)}{g(\sigma \tau)}=\frac{f(\sigma)}{f(\sigma \tau)}$$
and we deduce that
$$g(\sigma)=f(\sigma) \Rightarrow g(\sigma \tau)=f(\sigma\tau).$$
Since $g(\id)=f(\id)$, we deduce by induction that $g$ and $f$ coincide on every product of transpositions, and we conclude that $g=f$.
\end{proof}

From there, we can derive Corollary \ref{fgtransfocor} from Theorem \ref{fgtransfotheo}.
Let $f$ and $g$ be mappings from $\mathfrak{S}_n$ to $\F^*$.
Let $U : \Mat_n(\F) \rightarrow \Mat_n(\F)$ be a linear mapping such that $U^{-1}(\calC(g))=\calC(f)$.
By Proposition \ref{fullnormalizedprop}, $f$ is PH-equivalent to a fully-normalized mapping $f'$, and
$g$ is PH-equivalent to a fully-normalized mapping $g'$. This yields linear automorphisms $V_1$ and $V_2$
of $\Mat_n(\F)$ such that $\widetilde{f}(V_1(M))=\widetilde{f'}(M)$ and $\widetilde{g'}(V_2(M))=\widetilde{g}(M)$ for all $M \in \Mat_n(\F)$.
Set $U':=V_2\circ U \circ V_1$. Then, $U'$ is an endomorphism of the vector space $\Mat_n(\F)$, and
$(U')^{-1}(\calC(g'))=\calC(f')$. Let us apply Theorem \ref{fgtransfotheo} to $U'$.
Assume first that
$$U' : M \mapsto K \star (P_\sigma V(M) P_\tau)$$
for some standard $f'$-similarity $V$, some permutations $\sigma$ and $\tau$, and some matrix $K \in \Mat_n(\F^*)$.
Note that we have a non-zero scalar $\alpha$ such that $\widetilde{f'}(V(M))=\alpha \widetilde{f'}(M)$ for all $M \in \Mat_n(\F)$.

The Schur functional
$$M \mapsto \widetilde{g'}\bigl(K \star (P_\sigma M P_\tau)\bigr)$$
then reads $\widetilde{h}$ for some mapping $h : \mathfrak{S}_n \rightarrow \F^*$
that is PH-equivalent to $g'$. Since $(U')^{-1}(\calC(g'))=\calC(f')$, we find that, for all $M \in \Mat_n(\F)$,
$$\widetilde{f'}(M)=0 \Leftrightarrow \widetilde{f'}(V^{-1}(M))=0
\Leftrightarrow \widetilde{g'}(U'(V^{-1}(M)))=0
\Leftrightarrow \widetilde{h}(M)=0.$$
In other words, $\calC(h)=\calC(f')$. We deduce from Proposition \ref{coneprop} that $h=\beta f'$ for some non-zero scalar $\beta$.
In particular, $h$ is PH-equivalent to $f'$ (see Remark \ref{HequivalentScalarProduct}), and by transitivity $g$ is PH-equivalent to $f$.

Moreover, for all $M \in \Mat_n(\F)$,
$$\widetilde{g'}(U'(M))=\widetilde{h}(V(M))=\beta \widetilde{f'}(V(M))=\alpha\beta \widetilde{f'}(M).$$
Coming back to the definition of $U$, we deduce that it is an $(\alpha\beta f,g)$-transformation.

Finally, if $U$ is of the type described in point (b) of Theorem \ref{fgtransfotheo}, then
$M \mapsto U(M^T)$ is of the type described in point (a) of that theorem for the pair $(f^T,g)$.
Applying the above results in that situation yields the claimed statements.

Therefore, Corollary \ref{fgtransfocor} is now established.

\subsection{Proof of Corollary \ref{fgtransfocor2}}\label{cor2section}

Let $f : \mathfrak{S}_n \rightarrow \F^*$ be a rigid mapping.

By Theorem \ref{fgtransfocor2}, there are a matrix $K \in \Mat_n(\F^*)$,
permutations $\sigma,\tau$ of $\lcro 1,n\rcro$, and a standard $f$-similarity $V$
such that $U : M \mapsto K \star (P_\sigma V(M)P_\tau)$ or
$U : M \mapsto K \star (P_\sigma V(M^T)P_\tau)$.

The rigidity of $f$ shows that $V$ reads $M \mapsto D M \Delta$ for invertible diagonal matrices $D,\Delta$ of $\Mat_n(\F)$.
Denoting by $X$ the diagonal vector of $D$ and by $Y$ the one of $\Delta$, we see that $L:=X Y^T$
has all its entries nonzero and $V: M \mapsto L \star M$.
Hence
$$\forall M \in \Mat_n(\F), \; U(M)=\bigl(K \star (P_\sigma L P_\tau)\bigr) \star (P_\sigma MP_\tau)$$
or
$$\forall M \in \Mat_n(\F), \; U(M)=\bigl(K \star (P_\sigma L P_\tau)\bigr) \star (P_\sigma M^TP_\tau).$$
Obviously, the matrix $K':=K \star (P_\sigma L P_\tau)$ has all its entries nonzero, and hence Corollary
\ref{fgtransfocor2} is proved.

\subsection{Application to the preservers of the determinant and of the permanent}\label{detperpressection}

Here, we show how Theorem \ref{fgtransfotheo} easily yields the non-trivial part in Frobenius's result (Theorem \ref{detpreservers}) and the permanent preservers (see Theorem \ref{perpreservers}) with no restriction of cardinality on the underlying field.

\subsubsection{Determinant preservers}

Here, we consider the case when $f$ is the signature mapping.
The known properties of the determinant show that $M \mapsto PMQ$ and
$M \mapsto PM^TQ$ are $(f,f)$-transformations for every pair
$(P,Q)\in \GL_n(\F)^2$ such that $\det P\det Q=1$. Conversely, let $U$ be an $(f,f)$-transformation. The only column-$f$-adapted permutation is the identity, ditto for row-$f$-adaptivity.
Moreover, the only super-$f$-adapted matrix is $E:=(1)_{1 \leq i,j \leq n}$.
Hence, there exists a standard $f$-similarity $V$ such that
$U : M \mapsto E \star V(M)=V(M)$ or  $U : M \mapsto E \star V(M^T)=V(M^T)$.
Write $V : M \mapsto PMQ$ for some matrices $P,Q$ of $\GL_n(\F)$.
Since the determinant is invariant under transposing, we obtain that $\det P \det Q=1$,
which yields the conclusion.

\subsubsection{Permanent preservers}

Here, we consider the case when $f$ is constant with value $1$, $\chi(\F) \neq 2$ and $n \geq 3$.
Let $\sigma$ and $\tau$ belong to $\mathfrak{S}_n$, and let $K \in \Mat_n(\F^*)$
be a normalized rank $1$ matrix.
Obviously, $M \mapsto P_\sigma M P_\tau$ and $M \mapsto P_\sigma M^T P_\tau$
are $(f,f)$-transformations. So is $M \mapsto K \star M$, by Lemma \ref{HadProductLemma}.
Hence, $M \mapsto K \star (P_\sigma M P_\tau)$ and $M \mapsto K \star (P_\sigma M^T P_\tau)$
are $(f,f)$-transformations.

Conversely, let $U$ be an $(f,f)$-transformation.
By Example \ref{perExample}, the mapping $f$ is rigid.
By Corollary \ref{fgtransfocor2}, there is a matrix $K \in \Mat_n(\F^*)$
together with permutations $\sigma,\tau$ of $\lcro 1,n\rcro$ such that
$$U: M \mapsto K \star (P_\sigma M P_\tau) \quad \text{or} \quad
U :M \mapsto K \star (P_\sigma M^T P_\tau).$$
Yet, we have seen earlier that $M \mapsto P_\sigma M P_\tau$ and $M \mapsto P_\sigma M^T P_\tau$ are $(f,f)$-transformations, and hence $M \mapsto K \star M$ is an $(f,f)$-transformation in any case. By Lemma \ref{HadProductLemma}, we conclude that $K$ is a normalized rank $1$ matrix.
This completes the proof of Theorem \ref{perpreservers} for an arbitrary field with characteristic not $2$.

\section{The case of central mappings}\label{centralsection}

\subsection{Preliminaries}

Let $n \geq 2$. In this section, we give an explicit description of all $(f,f)$-transformations when $f : \mathfrak{S}_n \rightarrow \F^*$
is a central mapping, i.e.\ $f(\tau \sigma \tau^{-1})=f(\sigma)$ for all $(\sigma,\tau)\in \mathfrak{S}_n^2$.
Fix such a mapping $f$.
First of all, it is obvious that $M \mapsto P_\sigma M P_\sigma^{-1}$
is an $(f,f)$-transformation for all $\sigma \in \mathfrak{S}_n$.
Next, every permutation of $\lcro 1,n\rcro$ is conjugated to its inverse, whence $f^T=f$.
It follows that $M \in \Mat_n(\F) \mapsto M^T$ is an $(f,f)$-transformation.

Next, we have seen in Proposition \ref{centralequivalencetheo} that either $f$ is H-equivalent to the signature, in which case the $(f,f)$-transformations are known, or $f$ is rigid.

In the rest of this section, we consider the case when $f$ is rigid.
By Corollary \ref{fgtransfocor2}, any $(f,f)$-transformation has one of the forms
$$M \mapsto A \star (P_\sigma MP_\tau) \quad \text{or} \quad M \mapsto A \star (P_\sigma M^T P_\tau)$$
for some matrix $A \in \Mat_n(\F^*)$ and some pair $(\sigma,\tau)\in (\mathfrak{S}_n)^2$,
and in that case we find that the mapping
$$M \mapsto A \star (M P_{\sigma \tau})$$
is an $(f,f)$-transformation because both $M \mapsto P_\sigma M P_\sigma^{-1}$ and
$M \mapsto P_\sigma M^T P_\sigma^{-1}$ are $(f,f)$-transformations.
In other words, we have, for the matrix $B:=(A P_{(\sigma \tau)^{-1}})^{[-1]}$, the identity
$$\forall M \in \Mat_n(\F), \; \widetilde{f}(MP_{\sigma\tau})=\widetilde{f}(B \star M).$$

This motivates the following definition, where we no longer discard the possibility that
$f$ be H-equivalent to the signature:

\begin{Def}
Let $f : \mathfrak{S}_n \rightarrow \F^*$ be a central mapping.

Let $\tau \in \mathfrak{S}_n$.
We say that $\tau$ if \textbf{$f$-coherent} whenever there exists a matrix $A \in \Mat_n(\F^*)$
(called \textbf{$\tau$-adapted}) such that
$$\forall M \in \Mat_n(\F), \; \widetilde{f}(MP_\tau)=\widetilde{f}(A \star M).$$
We denote by $G_f$ the set of all $f$-coherent permutations of $\mathfrak{S}_n$.
\end{Def}

For example, if $f$ is constant, then every permutation of $\mathfrak{S}_n$ is $f$-coherent.

Now, assume that $f$ is not H-equivalent to the signature, that we know the $f$-coherent permutations of $\lcro 1,n\rcro$ and, for each such permutation $\tau$, that we have an adapted matrix $A_\tau$. Then, one sees that the set of all
$(f,f)$-transformations is the set of all maps having one of the forms
$$M \mapsto (K \star (A_\tau P_\tau)^{[-1]})\star (P_{\sigma} M P_{\sigma^{-1}\tau})$$
or
$$M \mapsto (K \star (A_\tau P_\tau)^{[-1]})\star (P_{\sigma} M^T P_{\sigma^{-1}\tau}),$$
where $\tau \in G_f$, $\sigma \in \mathfrak{S}_n$, and $K \in \Mat_n(\F^*)$ is a normalized rank $1$ matrix.

Hence, it remains to understand what the set $G_f$ can be and, for each $\tau$ in $G_f$, to find an adapted matrix $A_\tau$.

Two final remarks before we begin our study: let $(A,\tau) \in \Mat_n(\F^*) \times \mathfrak{S}_n$.
Then, the identity
$$\forall M\in \Mat_n(\F), \; \widetilde{f}(MP_\tau)=\widetilde{f}(A \star M)$$
is equivalent to
$$\forall \sigma \in \mathfrak{S}_n, \; f(\sigma\tau)=f(\sigma)\prod_{j=1}^n a_{\sigma(j),j.}$$
Moreover, if this property is satisfied by the pair $(A,\tau)$ then it is also satisfied by
$(K \star A,\tau)$ for every normalized rank $1$ matrix $K$ in $\Mat_n(\F)$.

\subsection{The set of all $f$-coherent permutations is a normal subgroup}\label{normalsubgroupsection}

\begin{prop}\label{normalsubgroupprop}
Let $f : \mathfrak{S}_n \rightarrow \F^*$.
Then, $G_f$ is a normal subgroup of $\mathfrak{S}_n$.
\end{prop}

\begin{proof}
Obviously the identity of $\lcro 1,n\rcro$ is $f$-coherent and the matrix
of $\Mat_n(\F)$ with all entries equal to $1$ is adapted to it.
Next, let $\sigma,\tau$ be $f$-coherent permutations, with respective adapted matrices $A,B$.
Then, for all $M \in \Mat_n(\F)$,
\begin{align*}
\widetilde{f}(MP_{\sigma\tau})& =\widetilde{f}(B \star (MP_\sigma)) \\
& =\widetilde{f}\bigl((B P_\sigma^{-1}) \star M)P_\sigma\bigr) \\
& =\widetilde{f}\bigl(A \star (B P_\sigma^{-1}) \star M).
\end{align*}
Hence, $\sigma \tau$ is $f$-coherent and $A \star (B P_\sigma^{-1})$ is an adapted matrix.
It follows that every positive power of $\sigma$ belongs to $G_f$, and since $\sigma$ has finite order
we get that $\sigma^{-1}$ belongs to $G_f$. Hence $G_f$ is a subgroup of $\mathfrak{S}_n$.

Finally, let $u \in \mathfrak{S}_n$, and let $(\tau,A) \in G_f \times \Mat_n(\F^*)$
be such that
$$\forall M \in \Mat_n(\F), \; \widetilde{f}(MP_\tau)=\widetilde{f}(A \star M).$$
Then, for all $M \in \Mat_n(\F)$,
\begin{align*}
\widetilde{f}(MP_{u\tau u^{-1}})
& = \widetilde{f}(P_u (P_u^{-1}MP_u)P_\tau P_{u^{-1}}) \\
& =\widetilde{f}\bigl((P_u^{-1}MP_u)P_\tau\bigr) \\
& =\widetilde{f}\bigl(A \star (P_u^{-1}MP_u)\bigr) \\
& =\widetilde{f}\bigl(P_u(A \star (P_u^{-1}MP_u))P_u^{-1}\bigr) \\
& =\widetilde{f}\bigl((P_uAP_u^{-1}) \star M\bigr).
\end{align*}
Hence, $u \tau u^{-1} \in G_f$. We conclude that $G_f$ is a normal subgroup of $\mathfrak{S}_n$.
\end{proof}

The normal subgroups of $\mathfrak{S}_n$ are well-known.
Hence, we have either $G_f=\{\id\}$, or $G=\mathfrak{A}_n$, or $G_f=\mathfrak{S}_n$,
or $n=4$ and $G_f$ is the Klein group
$$K_4:=\bigl\{\id,\tau_{1,2}\tau_{3,4},\tau_{1,3}\tau_{2,4},\tau_{1,4}\tau_{2,3}\bigr\}.$$

In the remainder of the article, we characterize, for each normal subgroup
$H \vartriangleleft \mathfrak{S}_n$, the central mappings $f$ for which $H \subset G_f$, and for each
such map and each $h \in H$ we give an adapted matrix.
From those classifications, it is easy to derive the one of the maps for which $H=G_f$.
We start with a technical result.

\begin{prop}\label{centralequivalenceprop}
Let $f : \mathfrak{S}_n  \rightarrow \F^*$ be a central mapping.
Let $\alpha,\beta$ belong to $\F^*$, and define
$B:=(b_{i,j})_{1 \leq i,j \leq n}$ by $b_{i,j}:=\alpha$ if $i=j$, and $b_{i,j}:=1$ otherwise, and
$C:=(c_{i,j})_{1 \leq i,j \leq n}$ by $c_{i,j}:=\beta$ if $i=1$, and $c_{i,j}:=1$ otherwise.
Then, the mapping
$$g : \sigma \mapsto \alpha^{\nfix(\sigma)} \beta f(\sigma)$$
is central and
$$\forall M \in \Mat_n(\F), \; \widetilde{g}(M)=\widetilde{f}((B \star C)\star M).$$
Moreover, $G_g=G_f$ and, for every $f$-coherent permutation $\sigma$, with corresponding adapted matrix $A$,
the matrix $B^{[-1]}\star C^{[-1]} \star A \star (BP_\sigma^{-1}) \star (CP_\sigma^{-1})$ is adapted to $\sigma$ as a $g$-coherent permutation.
\end{prop}

\begin{proof}
The first two statements are obvious.
Now, let $(A,\sigma) \in \Mat_n(\F^*) \times \mathfrak{S}_n$
be such that $\widetilde{f}(MP_\sigma)=\widetilde{f}(A \star M)$ for all $M \in \Mat_n(\F)$.
Set $D:=B \star C$.
Then, for all $M \in \Mat_n(\F)$,
\begin{align*}
\widetilde{g}(MP_\sigma) & = \widetilde{f}(D \star (MP_\sigma)) \\
& = \widetilde{f}\bigl((DP_\sigma^{-1}) \star M)P_\sigma\bigr) \\
& = \widetilde{f}\bigl(A \star (DP_\sigma^{-1}) \star M\bigr) \\
& = \widetilde{g}\bigl(D^{[-1]} \star A \star (DP_\sigma^{-1}) \star M\bigr),
\end{align*}
which yields the last statement.
\end{proof}

\begin{Def}
Two central mappings $f$ and $g$ from $\mathfrak{S}_n$ to $\F^*$
are called \textbf{centrally equivalent} when there exists a pair $(\alpha,\beta)\in (\F^*)^2$
such that
$$\forall \sigma \in \mathfrak{S}_n, \; g(\sigma)=\alpha^{\nfix(\sigma)} \beta f(\sigma).$$
\end{Def}

Obviously, this defines an equivalence relation on the set of all central mappings from $\mathfrak{S}_n$ to $\F^*$. As we have seen, two centrally equivalent central mappings are necessarily H-equivalent.
The converse fails for $n=2$: in that case indeed any two mappings $f$, $g$ are H-equivalent, but they are centrally equivalent if and only if $\frac{g(\tau_{1,2})}{g(\id)}\,\frac{f(\id)}{f(\tau_{1,2})}$ is a square in $\F^*$,
which might not be true over general fields.

\begin{lemma}\label{constanttypelemma}
Let $H$ be a normal subgroup of $\mathfrak{S}_n$, and let
$f : \mathfrak{S}_n \rightarrow \F^*$ be a central mapping that is constant on every class in
$\mathfrak{S}_n/H$. Then, $H \subset G_f$, and for all $\sigma \in H$ the matrix $E:=(1)_{1 \leq i,j \leq n}$
is $\sigma$-adapted.
\end{lemma}

\begin{proof}
Let $\tau \in H$. For all $\sigma \in \mathfrak{S}_n$, the assumptions show that $f(\sigma \tau)=f(\sigma)$, whence
$\tau \in G_f$ and $E$ is $\tau$-adapted.
\end{proof}

Let us set aside the trivial case when $H=\{\id\}$ or $H=\mathfrak{S}_n$.
Characterizing the maps that satisfy the assumption of Lemma \ref{constanttypelemma} is easy:
\begin{itemize}
\item If $H=\mathfrak{A}_n$ then the central maps $f : \mathfrak{S}_n \rightarrow \F^*$ that are constant on each class in $\mathfrak{S}_n/H$ are the maps that are constant on $\mathfrak{A}_n$ and on $\mathfrak{S}_n \setminus \mathfrak{A}_n$.
\item If $H=K_4$ then the central maps $f : \mathfrak{S}_4 \rightarrow \F^*$ that are constant on each class in $\mathfrak{S}_4/H$ are the central maps that give the same value to the $4$-cycles and to the transpositions, and that
    give the same value to the double-transpositions and to the identity.
\end{itemize}

In the remainder of the section, we examine converse statements.
Here are our two main results:

\begin{prop}\label{GfSn}
Let $n \geq 3$ and $f : \mathfrak{S}_n \rightarrow \F^*$ be a central map.
Then, $G_f=\mathfrak{S}_n$ if and only if $f$ is centrally equivalent to the constant map with value $1$
or to the signature.
\end{prop}

\begin{prop}\label{GfAn}
Let $n \geq 4$ and $f : \mathfrak{S}_n \rightarrow \F^*$ be a central map.
Then, $\mathfrak{A}_n \subset G_f$ if and only if $f$ is centrally equivalent to a map
that is constant on $\mathfrak{A}_n$ and on $\mathfrak{S}_n \setminus \mathfrak{A}_n$.
\end{prop}

The latter result fails for $n=3$, as we will see in the next section.

The rest of the section is laid out as follows:
\begin{itemize}
\item In Section \ref{n=3section}, we tackle the case $n=3$.
\item In Section \ref{GfAnsection}, we prove Proposition \ref{GfAn} by induction on $n$.
\item In Section \ref{GfSnsection}, we derive Proposition \ref{GfSn} from Proposition \ref{GfAn}.
\item In Section \ref{GfK4section}, which is logically independent from the other ones, we determine which central mappings $f$ satisfy $K_4 \subset G_f$.
\end{itemize}

\subsection{The case $n=3$}\label{n=3section}

\begin{prop}\label{GfAn3}
Let $f : \mathfrak{S}_3 \rightarrow \F^*$.
Then, $\mathfrak{A}_3 \subset G_f$.
Moreover, if we set $\alpha:=\frac{f((1\,2\,3))}{f(\id)}$, then
the matrix $A:=\begin{bmatrix}
1 & \alpha^{-1} & 1 \\
1 & 1 & 1 \\
1 & 1 & \alpha
\end{bmatrix}$ is adapted to $(1\,2\,3)$.
\end{prop}

\begin{proof}
For all $\sigma\in \mathfrak{S}_3$,
one checks that
$$\prod_{j=1}^3 a_{\sigma(j),j}=\begin{cases}
1 & \text{if $\sigma$ is a transposition or $\sigma=(1\,2\,3)$} \\
\alpha^{-1} & \text{if $\sigma=(1\,3\,2)$} \\
\alpha & \text{if $\sigma=\id$,}
\end{cases}$$
whereas, as $f$ is central,
$$f\bigl(\sigma\,(1\,2\,3)\bigr)=\begin{cases}
f(\sigma) & \text{if $\sigma$ is a transposition} \\
\alpha\,f(\sigma) & \text{if $\sigma=\id$} \\
f(\sigma) & \text{if $\sigma=(1\,2\,3)$} \\
\alpha^{-1}f(\sigma) & \text{if $\sigma=(1\,3\,2)$.}
\end{cases}$$
This yields the claimed statement.
\end{proof}

\begin{lemma}\label{GfSn3}
Let $f : \mathfrak{S}_3 \rightarrow \F^*$ be a central mapping.
Then, $G_f=\mathfrak{S}_3$ if and only if $f$
is centrally equivalent to the signature or to the constant map with value $1$.
\end{lemma}

\begin{proof}
The converse statement is already known. Assume that $G_f=\mathfrak{S}_3$.
In particular, $\tau_{1,2}$ belongs to $G_f$, which yields a matrix $A \in \Mat_3(\F^*)$
such that $f(\sigma \tau_{1,2})=\underset{j=1}{\overset{3}{\prod}} a_{\sigma(j),j}\, f(\sigma)$
for all $\sigma \in \mathfrak{S}_3$. Setting $\Pi:=\underset{1 \leq i,j \leq 3}{\prod} a_{i,j}$,
we deduce that
$$\prod_{\sigma \in \mathfrak{A}_3} f(\sigma \tau_{1,2})=\Pi \prod_{\sigma \in \mathfrak{A}_3} f(\sigma)
\qquad \text{and} \quad
\prod_{\sigma \in \mathfrak{S}_3 \setminus \mathfrak{A}_3} f(\sigma \tau_{1,2})=\Pi \prod_{\sigma \in \mathfrak{S}_3 \setminus \mathfrak{A}_3} f(\sigma),$$
whence
$$\prod_{\sigma \in \mathfrak{S}_3 \setminus \mathfrak{A}_3} f(\sigma)=\Pi \prod_{\sigma \in \mathfrak{A}_3} f(\sigma) \qquad \text{and} \quad
\prod_{\sigma \in \mathfrak{A}_3} f(\sigma)=\Pi \prod_{\sigma \in \mathfrak{S}_3 \setminus \mathfrak{A}_3} f(\sigma).$$
It follows that $\Pi^2=1$, whence $\Pi=\pm 1$.
Set $a:=f(\id)$, $b:=f(\tau_{1,2})$ and $c:=f((1\,2\,3))$.
Then, the above shows that $b^3=\varepsilon ac^2$ for some $\varepsilon \in \{1,-1\}$.
Setting $\alpha:=\varepsilon\,\frac{b}{c}$, we obtain $a=c \alpha^3$ and $b=\varepsilon c \alpha$.
Therefore:
\begin{itemize}
\item if $\varepsilon=1$ then
$f(\sigma)=c \alpha^{\nfix(\sigma)}$ for all $\sigma \in \mathfrak{S}_3$;
\item if $\varepsilon=-1$ then
$f(\sigma)=c \alpha^{\nfix(\sigma)} \sgn(\sigma)$ for all $\sigma \in \mathfrak{S}_3$.
\end{itemize}
Hence, $f$ is centrally equivalent to the constant map with value $1$ or to the signature.
\end{proof}

\subsection{Maps for which $\mathfrak{A}_n \subset G_f$}\label{GfAnsection}

Here, we prove Proposition \ref{GfAn}. We start with the case $n=4$.

\begin{lemma}\label{GfAn4}
Let $f : \mathfrak{S}_4 \rightarrow \F^*$ be a central map such that
$\mathfrak{A}_4 \subset G_f$. Then, $f$ is centrally equivalent to a map that is constant on $\mathfrak{A}_4$ and on
$\mathfrak{S}_4 \setminus \mathfrak{A}_4$.
\end{lemma}

\begin{proof}
Set $\alpha:=\frac{f((1\,2\,3))}{f(\tau_{1,2}\tau_{3,4})}\cdot$
Then, by replacing $f$ by $\sigma \mapsto f(\id)^{-1} f(\sigma)\, \alpha^{4-\nfix(\sigma)}$,
we lose no generality in assuming that
$$f(\tau_{1,2}\tau_{3,4})=f\bigl((1\,2\,3)\bigr) \quad \text{and} \quad
f(\id)=1.$$
Set $a:=f\bigl((1\,2\,3)\bigr)$, $b:=f(\tau_{1,2})$ and $c:=f\bigl((1\,2\,3\,4)\bigr)$.
We shall prove that $a=1$ and $b=c$. Since $f$ is central, this will prove that $f$ is constant with value $1$ on $\mathfrak{A}_4$,
and constant with value $b$ on $\mathfrak{S}_4$.

Set $K:=\begin{bmatrix}
1 & a^{-1} & 1 \\
1 & 1 & 1 \\
1 & 1 & a
\end{bmatrix}$.

We know that $\tau:=(1\,2\,3)$ is $f$-coherent, and we choose an adapted matrix $A=(a_{i,j})$.
Multiplying $A$ with a well-chosen normalized rank $1$ matrix, we can assume that $a_{4,4}=1$.
Given $\sigma \in \mathfrak{S}_3$, denote by $\overline{\sigma}$ its extension as a permutation of $\lcro 1,4\rcro$, and set
$\overline{f} : \sigma \in \mathfrak{S}_3 \mapsto f(\overline{\sigma})$, which is obviously central.
Since $a_{4,4}=1$, we find that the submatrix $B:=(a_{i,j})_{1 \leq i,j \leq 3}$ satisfies
$$\forall \sigma \in \mathfrak{S}_3, \; \overline{f}\bigl(\sigma (1\,2\,3)\bigr)=f(\sigma) \prod_{j=1}^3 a_{\sigma(j),j}.$$
Then, by Proposition \ref{GfAn3}, we find that $B$ equals $L \star K$ for some normalized rank $1$
matrix $L \in \Mat_3(\F^*)$. Write $L=XY^T$ where $X$ and $Y$ belong to $(\F^*)^3$, and extend $X$ and $Y$ to
vectors $\widetilde{X}$ and $\widetilde{Y}$ of $(\F^*)^4$ by taking the last entry equal to $1$. Set
$\widetilde{L}:=\widetilde{X} \widetilde{Y}^T$, which is a normalized rank $1$ matrix.
Then, replacing $A$ with $\widetilde{L}^{[-1]}\star A$, we reduce the situation to the one where
$$A=\begin{bmatrix}
1 & a^{-1} & 1 & x' \\
1 & 1 & 1 & y' \\
1 & 1 & a & z' \\
x & y & z & 1
\end{bmatrix}$$
for some non-zero scalars $x,x',y,y',z,z'$.

For every double-transposition $\sigma$, we see that $\sigma \tau$ is a $3$-cycle (it is not in the Klein subgroup, yet it has signature $1$),
whence $a\underset{j=1}{\overset{4}{\prod}} a_{\sigma(j),j}=a$, leading to $\underset{j=1}{\overset{4}{\prod}} a_{\sigma(j),j}=1$.
Taking all possible $\sigma$'s leads to $zz'=a$, $yy'=1$ and $xx'=1$.

On the other hand, for all $\sigma$ in $\bigl\{(4\,1\,3),(4\,2\,1),(4\,3\,2)\bigr\}$, one sees
that $\sigma \tau$ is a $3$-cycle, and one deduces that $zx'=1$, $xy'=1$ and $yz'=1$.
Taking the product yields $(xx')(yy')(zz')=1$, and hence $a=1$.

Finally, since the composite $\tau_{1,4} \tau$ is a $4$-cycle, we find
$\frac{c}{b}=axx'=1$, and hence $c=b$. This completes the proof.
\end{proof}

Now, we are ready to prove Proposition \ref{GfAn}.
The proof works by induction on $n$.
The case $n=4$ has already been dealt with.
Let now $n>4$, and let $f : \mathfrak{S}_n \rightarrow \F^*$ be a central map.
If $f$ is centrally equivalent to a map that is constant on $\mathfrak{A}_n$ and $\mathfrak{S}_n \setminus \mathfrak{A}_n$,
then we already know that $\mathfrak{A}_n \subset G_f$. Conversely, we assume that $\mathfrak{A}_n \subset G_f$.

As in the proof of Lemma \ref{GfAn4}, we lose no generality in assuming that $f$ maps $3$-cycles and double-transpositions to the same value.
Next, the $3$-cycle $\tau:=(1\,2\,3)$ is $f$-coherent, and we choose a corresponding matrix $A=(a_{i,j}) \in \Mat_n(\F^*)$.
Just like in the proof of Lemma \ref{GfAn4}, we can assume that $a_{n,n}=1$.
Then, we define the central mapping $\overline{f} : \sigma \in \mathfrak{S}_{n-1} \mapsto f(\overline{\sigma})$
and we see that the $3$-cycle $(1\,2\,3)$ of $\mathfrak{S}_{n-1}$ belongs to $G_{\overline{f}}$. Hence, $G_{\overline{f}}$ is a normal subgroup of $\mathfrak{S}_{n-1}$
that contains a $3$-cycle, and it follows that this subgroup includes $\mathfrak{A}_{n-1}$. By induction, $\overline{f}$
is centrally equivalent to a map that is constant on $\mathfrak{A}_{n-1}$ and on $\mathfrak{S}_{n-1} \setminus \mathfrak{A}_{n-1}$.

\begin{claim}
The mapping $\overline{f}$ is constant on $\mathfrak{A}_{n-1}$ and on $\mathfrak{S}_{n-1} \setminus \mathfrak{A}_{n-1}$.
\end{claim}

\begin{proof}
We have non-zero scalars $a,b,c$ such that
$\overline{f}(\sigma)=b a^{\nfix(\sigma)}$ for all $\sigma \in \mathfrak{A}_{n-1}$, and
$\overline{f}(\sigma)=c a^{\nfix(\sigma)}$ for all $\sigma \in \mathfrak{S}_{n-1} \setminus \mathfrak{A}_{n-1}$.
Since $\overline{f}\bigl((1\,2\,3)\bigr)=\overline{f}(\tau_{1,2}\tau_{3,4})$, we find $a=1$, and the
conclusion follows.
\end{proof}

It follows that $f$ is constant on the set of all elements of $\mathfrak{A}_n$ that have a fixed point,
and constant on the set of all elements of $\mathfrak{S}_n \setminus \mathfrak{A}_n$ that have a fixed point.

Set $B:=(a_{i,j})_{1 \leq i,j \leq n-1}$. With the same line of reasoning as in the proof of Lemma \ref{GfAn4},
we find that $(1\,2\,3)$ is $\overline{f}$-coherent and that $B$ is adapted to it.
Yet, by the above claim, the matrix $(1)_{1 \leq i,j \leq n-1}$ is obviously adapted to
$(1\,2\,3)$. Hence, $B$ is a normalized rank $1$ matrix. With the same proof as for Lemma \ref{GfAn4},
we deduce that no generality is lost in assuming that $B=(1)_{1 \leq i,j \leq n-1}$.

Now, as $n \geq 5$, we see that for $u:=\tau_{n-1,n}$, we have, for all $M \in \Mat_n(\F)$,
$$\widetilde{f}(M P_\tau)=\widetilde{f}(P_u^{-1}M P_\tau P_u)=
\widetilde{f}(P_u^{-1} M P_u P_\tau)=\widetilde{f}(A \star (P_u^{-1} MP_u))
=\widetilde{f}((P_u A P_u^{-1}) \star M).$$
Hence $P_u A P_u^{-1}=K \star A$ for some normalized rank $1$ matrix $K$.
In particular, the matrix obtained from $A$ by deleting the $(n-1)$-th row and column
has rank $1$. It follows that $a_{1,n}=\cdots=a_{n-2,n}$, $a_{n,1}=\cdots=a_{n,n-2}$
and $a_{1,n}=a_{n,1}^{-1}$. Setting $d:=a_{1,n}$, $X:=\begin{bmatrix}
1 & \cdots & 1 & d
\end{bmatrix}^T$, $Y:=\begin{bmatrix}
1 & \cdots & 1 & d^{-1}
\end{bmatrix}^T$ and replacing $A$ with $(XY^T) \star A$, we see that no generality is lost in further assuming that $a_{n,1}=\dots=a_{n,n-2}=1=a_{1,n}=\cdots=a_{n-2,n}$.

\begin{claim}
One has $a_{i,j}=1$ for all $i,j$ in $\lcro 1,n\rcro$.
\end{claim}

\begin{proof}
It only remains to prove that $a_{n,n-1}=a_{n-1,n}=1$.
Setting $c:=(1\,2\,\cdots n)$, we see that $c (1\,2\,3)=\bigl(1\,3\,2\,4\,5\,\cdots\,(n-1)\,n\bigr)$
is an $n$-cycle, whence $f\bigl(c\, (1\, 2\,3)\bigr)=f(c)$, and it ensues that $a_{n,n-1}=1$.
Likewise, with $d:=\bigl(1\,2\,\cdots (n-2)\,n \,(n-1)\bigr)$, we have
$d\, (1\,2\,3)=\bigl(1\,3\,2\,4\,5\,\cdots\,(n-2)\,n\,(n-1)\bigr)$ and we deduce that $a_{n-1,n}=1$.
\end{proof}

Now, we can conclude.
The above claim yields $f(\sigma \tau)=f(\sigma)$ for all $\sigma \in \mathfrak{S}_n$.
Since $f$ is central, we deduce that, for all $\sigma \in \mathfrak{S}_n$ and all $u \in \mathfrak{S}_n$,
$$f(\sigma (u\tau u^{-1}))=f(u^{-1} (\sigma u) \tau)=f(u^{-1} \sigma u)=f(\sigma).$$
Hence, $f(\sigma c)=f(\sigma)$ for every $3$-cycle $c \in \mathfrak{S}_n$ and every $\sigma \in \mathfrak{S}_n$.
Since $\mathfrak{A}_n$ is generated by the $3$-cycles, we conclude that $f$ is constant on
$\mathfrak{A}_n$ and on $\mathfrak{S}_n \setminus \mathfrak{A}_n$, which completes the proof of Proposition
\ref{GfAn}.

\subsection{Maps for which $G_f=\mathfrak{S}_n$}\label{GfSnsection}

Here, we derive Proposition \ref{GfSn} from Proposition \ref{GfAn}.
Let $f : \mathfrak{S}_n \rightarrow \F^*$ be a central map, with $n \geq 3$.
If $f$ is centrally equivalent to the signature or to a constant map,
then we already know that $G_f=\mathfrak{S}_n$.
Conversely, assume that $G_f=\mathfrak{S}_n$ and let us prove that $f$
is centrally equivalent to the signature or to the constant map with value $1$.
If $n=3$, this result is known by Lemma \ref{GfSn3}.
Assume now that $n \geq 4$. Then, Proposition \ref{GfAn} shows that $f$
is centrally equivalent to a map that is constant on $\mathfrak{A}_n$ and on $\mathfrak{S}_n \setminus \mathfrak{A}_n$.
Hence, no generality is lost in assuming that $f$ is constant on $\mathfrak{A}_n$ and on $\mathfrak{S}_n \setminus \mathfrak{A}_n$. Denote by $\alpha$ and $\beta$ the respective values of $f$ on $\mathfrak{A}_n$ and on $\mathfrak{S}_n \setminus \mathfrak{A}_n$.
The transposition $\tau:=\tau_{1,2}$ is $f$-coherent, and we choose an adapted matrix $A=(a_{i,j})$.
It follows that
$$\forall \sigma \in \mathfrak{A}_n, \; \beta=\alpha \prod_{j=1}^n a_{\sigma(j),j}
\quad \text{and} \quad
\forall \sigma \in \mathfrak{S}_n \setminus \mathfrak{A}_n, \; \alpha=\beta \prod_{j=1}^n a_{\sigma(j),j.}$$
Setting $t:=\frac{\beta}{\alpha}$, it follows that
$\underset{j=1}{\overset{n}{\prod}} a_{\sigma(j),j}$ equals $t$ if $\sgn(\sigma)=1$, and $t^{-1}$ otherwise.
Then,
$$\forall \sigma \in \mathfrak{S}_n, \; \frac{\underset{j=1}{\overset{n}{\prod}} a_{(\sigma \tau_{1,2})(j),j}}
{\underset{j=1}{\overset{n}{\prod}} a_{\sigma(j),j}}=t^{-2\sgn(\sigma)}$$
i.e.
$$\forall \sigma \in \mathfrak{S}_n, \; \frac{a_{\sigma(2),1}a_{\sigma(1),2}}{a_{\sigma(1),1}a_{\sigma(2),2}}
=t^{-2\sgn(\sigma)}.$$
Now, let $i,j$ be distinct elements of $\lcro 1,n\rcro$. Since $n \geq 4$, there
exist $\sigma \in \mathfrak{A}_n$ and $\sigma' \in \mathfrak{S}_n \setminus \mathfrak{A}_n$
such that $\sigma(1)=\sigma'(1)=i$ and $\sigma(2)=\sigma'(2)=j$, whence
$$\frac{a_{i,2}}{a_{i,1}}=t^2 \frac{a_{j,2}}{a_{j,1}}.$$
It follows in particular that
$$\frac{a_{1,2}}{a_{1,1}}=t^2 \frac{a_{2,2}}{a_{2,1}}=t^4 \frac{a_{3,2}}{a_{3,1}}=t^2 \frac{a_{1,2}}{a_{1,1}}$$
and we deduce that $t^2=1$. If $t=1$, then $f$ is constant with value $\alpha$,
otherwise $f=\alpha \sgn$.
We conclude that $f$ is centrally equivalent to the signature or to the constant map with value $1$.

\subsection{Maps for which $K_4 \subset G_f$}\label{GfK4section}

Here, we finish our study by characterizing the central maps $f : \mathfrak{S}_4 \rightarrow \F^*$ for which $K_4 \subset G_f$.

\begin{prop}\label{GfK4}
Let $f : \mathfrak{S}_4 \rightarrow \F^*$ be a central map, and set
$\alpha:=\frac{f(\tau_{1,2})}{f\bigl((1\,2\,3\,4)\bigr)}\cdot$
Then, $K_4 \subset G_f$ if and only if $f(\id)=\alpha^2 f(\tau_{1,2}\tau_{3,4})$.
Moreover, in that case the matrix
$$A:=\begin{bmatrix}
1 & \alpha & 1 & 1 \\
\alpha & 1 & 1 & 1 \\
1 & 1 & \alpha^{-1} & 1 \\
1 & 1 & 1 & \alpha^{-1}
\end{bmatrix}$$
is adapted to the double-transposition $\tau_{1,2}\tau_{3,4}$.
\end{prop}

Using the line of reasoning from the end of the proof of Proposition \ref{normalsubgroupprop},
it is then easy to find an adapted matrix for each double-transposition (we leave this mundane task to the reader).

Note that the condition given here is satisfied if $f$ is centrally equivalent to a (central)
map that takes the same value at $\id$ and at double-transpositions, and that takes the same
value at transpositions and at $4$-cycles (see the end of Section \ref{normalsubgroupsection}).
However, the converse is not true over general fields: as an example, take a non-zero scalar
$\alpha \in \F^*$ that is not a square in $\F$, and define $f$
as the central map that takes the value $\alpha^2$ at $\id$, the value $1$ at double-transpositions,
$3$-cycles and $4$-cycles, and the value $\alpha$ at transpositions.

\begin{proof}[Proof of Proposition \ref{GfK4}]
Set $\tau:=\tau_{1,2} \tau_{3,4}$.

Assume first that $K_4 \subset G_f$. Then, $\tau$
is $f$-coherent, and we choose an adapted matrix $B=(b_{i,j}) \in \Mat_4(\F^*)$.
Note that $\sigma \tau$ is a $3$-cycle for every $3$-cycle $\sigma \in \mathfrak{S}_4$
(because it belongs to $\mathfrak{A}_4 \setminus K_4$).
Hence, $\underset{j=1}{\overset{4}{\prod}} b_{\sigma(j),j}=1$ for every $3$-cycle $\sigma \in \mathfrak{S}_4$.
Taking the $3$-cycles $(1\,2\,3)$, $(1\,3\,4)$, $(1\,4\,2)$ and $(2\,4\,3)$, we deduce that the product
$\Pi=\underset{1 \leq i,j \leq 4}{\prod} b_{i,j}$ equals $1$.
Yet, for $\sigma:=(1\,2\,3\,4)$, we find
$$\Pi:=\underset{k=0}{\overset{3}{\prod}} \prod_{j=1}^4 b_{\sigma^k(j),j}
=\frac{f(\tau)}{f(\id)}\,\frac{f(\sigma \tau)}{f(\sigma)}\,\frac{f(\sigma^2 \tau)}{f(\sigma^2)}
\,\frac{f(\sigma^3 \tau)}{f(\sigma^3)}\cdot$$
One checks that $\sigma \tau$ and $\sigma^3 \tau$ are transpositions, whereas
$\sigma^3$ and $\sigma$ are $4$-cycles, and $\sigma^2 \tau$ and $\sigma^2$ are double-transpositions.
Hence,
$$\frac{f(\tau)}{f(\id)}\,\frac{f(\sigma \tau)^2}{f(\sigma)^2}=1,$$
which shows that $f(\id)=\alpha^2 f(\tau)$.

Conversely, assume that  $f(\id)=\alpha^2 f(\tau)$.
If we prove that $\tau \in G_f$, then every double-transposition will belong to $G_f$ because $G_f$
is a normal subgroup of $\mathfrak{S}_4$, and we will conclude that $K_4 \subset G_f$.
Hence, it suffices to prove that $\tau \in G_f$.
To do so, we will prove that
$$\forall M \in \Mat_4(\F), \; \widetilde{f}(MP_\tau)=\widetilde{f}(A \star M).$$
This can be proved by a tedious computation, but we will give a more satisfying proof.
Let us choose an extension $\mathbb{L}$ of the field $\F$ in which $\alpha$ has a square-root
$\delta$. Denote by $\widetilde{f}_\L$ the Schur functional on $\Mat_4(\L)$ associated with $f$.
Set $g : \sigma \in \mathfrak{S}_4 \mapsto f(\sigma)\, \delta^{-\nfix(\sigma)} \in \L^*$,
which is centrally equivalent to $f$ with respect to the field $\L$.
Noting that $f(\id)=\delta^4 f(\tau_{1,2}\tau_{3,4})$ and $f(\tau_{1,2})=\delta^2f\bigl((1\,2\,3\,4)\bigr)$,
we obtain that $g$ maps $\id$ and all double-transpositions to the same value in $\L^*$,
and maps all transpositions and all $4$-cycles to the same value in $\L^*$.
It then follows from Lemma \ref{constanttypelemma} that $\tau$ is $g$-coherent
and that the matrix $E$ of $\Mat_4(\L)$ with all entries equal to $1$ is adapted to $\tau$.
Denote by $B$ the matrix of $\Mat_4(\L)$ with all diagonal entries equal to $\delta$
and all off-diagonal entries equal to $1$.
Then, we see from Proposition \ref{centralequivalenceprop} that
the matrix
$$A':=B^{[-1]}\star (BP_\tau^{-1})$$
satisfies
$$\forall M \in \Mat_4(\L), \; \widetilde{f}_\L(MP_\tau)=\widetilde{f}_\L(A' \star M).$$
One computes that
$$A'=\begin{bmatrix}
\delta^{-1} & \delta & 1 & 1 \\
\delta & \delta^{-1} & 1 & 1 \\
1 & 1 & \delta^{-1} & \delta \\
1 & 1 & \delta & \delta^{-1}
\end{bmatrix}.$$
Setting
$$X:=\begin{bmatrix}
\delta & \delta & 1 & 1
\end{bmatrix}^T \quad \text{and} \quad Y:=\begin{bmatrix}
1 & 1 & \delta^{-1} & \delta^{-1}
\end{bmatrix}^T,$$
we see that $XY^T$ is a normalized rank $1$ matrix of $\Mat_4(\L)$, and we compute that
$$(XY^T) \star A'=A,$$
whence
$$\forall M \in \Mat_4(\L), \; \widetilde{f}_\L(MP_\tau)=\widetilde{f}_\L(A \star M).$$
In particular,
$$\forall M \in \Mat_4(\F), \; \widetilde{f}(MP_\tau)=\widetilde{f}(A \star M),$$
which completes the proof.
\end{proof}

\end{document}